\documentclass[11pt,reqno]{amsart}

\usepackage[table]{xcolor}
\usepackage[bmargin=2.5cm,tmargin = 3cm, hmargin=2cm]{geometry}
\usepackage{caption}

\usepackage[backref = page]{hyperref}
\hypersetup{
    colorlinks=true,
    linkcolor=blue,
    urlcolor =blue, 
    citecolor=magenta}

\usepackage{varwidth}

\renewcommand*{\backrefalt}[4]{%
\ifcase #1 %
\color{red} {No citations.}%
\or
(p.~#2).%
\else
(pp.~#2).%
\fi
}

\usepackage[nameinlink, capitalize, noabbrev]{cleveref}
\crefformat{app}{#2the Appendix#3}

\usepackage{amssymb,amsthm,amsmath, stmaryrd}

\usepackage{enumitem}
\setlist[enumerate,1]{label={\rm(\roman{enumi})}}
\setlist[enumerate,2]{label={\rm(\alph{enumii})}, ref=(\roman{enumi})-(\alph{enumii})}

\usepackage{mathtools}
\usepackage{multirow}
\usepackage{float}
\usepackage{cancel}

\usepackage{booktabs}

\colorlet{tblproven}{green!80!black}
\newcommand\proven{\rowcolor{tblproven}}
\newcommand{\modconja}{}
\newcommand{\modconj}{}

\usepackage{array}   % for \newcolumntype macro
\newcolumntype{C}{>{$}c<{$}} % math-mode version of "l" column type

\theoremstyle{plain}

\newtheorem{theorem}{Theorem}[section]
\newtheorem{thma}{Theorem}
\renewcommand{\thethma}{\Alph{thma}}
\crefname{thma}{Theorem}{Theorems}

\numberwithin{equation}{section}

\newtheorem{lemma}[theorem]{Lemma}
\newtheorem{prop}[theorem]{Proposition}

\newtheorem*{expectation}{Expectation}
\newtheorem{corollary}[theorem]{Corollary}

\theoremstyle{definition}
\newtheorem{remark}[theorem]{Remark}
\newtheorem{definition}[theorem]{Definition}

\newtheorem{example}[theorem]{Example}

\DeclareMathOperator{\Hom}{Hom}
\newcommand{\Z}{\mathbb{Z}}
\newcommand{\ZZ}{\Z}
\newcommand{\FF}{\mathbb{F}}
\newcommand{\Q}{\mathbb{Q}}
\newcommand{\QQ}{\Q}

\newcommand{\C}{\mathbb{C}}
\newcommand{\HH}{\mathbb{H}}
\newcommand{\llb}{\llbracket}
\newcommand{\rrb}{\rrbracket}

\DeclareMathOperator{\Gal}{Gal}
\DeclareMathOperator{\tr}{Tr}
\DeclareMathOperator{\GL}{GL}
\DeclareMathOperator{\SL}{SL}
\DeclareMathOperator{\Frob}{Frob}
\DeclareMathOperator{\ann}{ann}

\DeclareMathOperator{\End}{End}
\DeclareMathOperator{\Var}{Var}

\renewcommand{\aa}{\alpha}

\renewcommand{\Im}{\operatorname{Im}}

\newcommand{\mm}{{\mathfrak m}}
\renewcommand{\pmod}[1]{\mkern4mu(\mathrm{mod} \, #1)}

\newcommand{\cmod}[1]{\mkern4mu\mathrm{mod}\ #1}

\newcommand{\bel}{\delta}
\newcommand{\del}{D}

\makeatletter
\def\section{\@startsection{section}{1}%
  \z@{2.5\linespacing\@plus.5\linespacing\@minus.5\linespacing}{.5\linespacing}%
  {\normalfont\scshape\centering}}
\makeatother

\begin{document}

\title{On the parity of coefficients of eta powers}

\author{Steven Charlton}
\address{Max Planck Institute for Mathematics, Vivatsgasse 7, 53111 Bonn, Germany}
\email{charlton@mpim-bonn.mpg.de}

\author{Lukas Mauth}
\address{Department of Mathematics and Computer Science, Division of Mathematics, University of Cologne, \linebreak Weyertal 86-90, 50931 Cologne, Germany}
\email{lmauth@uni-koeln.de}

\author{Anna Medvedovsky}
\address{Department of Mathematics, University of Arizona, 617 N.~Santa Rita Ave, Tucson, AZ 85721, USA}
\email{medved@arizona.edu}



\begin{abstract} 
We consider a special subsequence of the Fourier coefficients of powers of the Dedekind~$\eta$-function, analogous to the sequence \( \delta_\ell \coloneqq 24^{-1} \pmod{\ell} \) on which exceptional congruences of the partition function are supported.
Therefrom we define a notion of density \( \del(r) \) for a normalized eta-power~$\eta^r$ measuring the proportion of primes $\ell$ for which the order at infinity of $U_\ell (\eta^r)$  modulo 2 is maximal. We relate $\del(r)$ to a notion of density measuring nonzero prime Fourier coefficients introduced by Bella\"iche, and use this to completely  
classify the vanishing of and establish upper bounds for~\( \del(r) \).
Furthermore, for several infinite families of \( \eta \) powers corresponding to dihedral/CM mod-2 modular forms 
in the sense of Nicholas-Serre and Bella\"iche, we explicitly compute the densities \( \del \).
We rely on Galois-theoretic techniques developed by Bellaïche in level 1 and extend these to level 9. En passant we take the opportunity to communicate proofs of two of Bellaïche's unpublished results on densities of mod-$2$ modular forms.
\end{abstract}

\maketitle

{
\parskip=0pt           % 15 pt spacing between paragraphs
	\setcounter{tocdepth}{1}
	\tableofcontents
	\setcounter{tocdepth}{2}
}

\section{Introduction}

For a non-negative integer \( n \) one defines \( p(n) \) as the number of non-increasing sequences of positive integers which sum to \( n \).  The value \( p(n) \) is the \emph{number of partitions} of \( n \), and \( p \)  is called the \emph{partition function}.  Understanding the arithmetic properties of \( p(n) \) is well-known to be a very challenging task, due to the lack of sufficient tools and machinery.  In particular, the Partition Parity Conjecture~\cite{parkin-shanks} asserts that the values \( p(n) \) are distributed equally between odd and even values, namely
\[
    \lim_{x \to \infty} \frac{\#\{ n \leq x \mid p(n) \equiv 0 \pmod{2} \}}{x} = \frac{1}{2} \,.
\]
One can begin to make progress in this question by studying the generating series of the partition function values,
\[
     \sum_{n=0}^\infty p(n) q^n = \prod_{n=1}^\infty \frac{1}{1 - q^n} = q^{1/24} \eta^{-1}(\tau) \,.
\]
Here \( \eta(\tau) \) is the Dedekind eta function defined, for \( \tau \) in the upper half-plane \( \HH \coloneqq \{ \tau \in \C \mid \Im{\tau} > 0 \} \), by the following infinite product
\[
    \eta(\tau) \coloneqq q^{1/24} \prod_{n=1}^\infty (1 - q^n), \quad q = \exp(2 \pi i \tau) \,.
\]
It is convenient to introduce 
\begin{equation}\label{P} P(q) \coloneqq \eta^{-1}(24 \tau) = \sum_{n=0}^\infty p(n) q^{24n-1}\,,\end{equation} where the rescaling naturally produces a Laurent series (i.e.~bounded-below integral powers of $q$), and the shift allows us to generalize directly to other eta-powers.  The Dedekind eta function is a modular form of weight~$\frac{1}{2}$, and up to a $q$-power, the partition function generating series is then a weakly holomorphic modular form of weight~$-\frac{1}{2}$.  In general understanding the coefficients of negative-weight, and half-integral--weight modular forms seems to be a difficult problem.  
We briefly survey progress on partition parity thus far. Already Kolberg~\cite{kolberg} established that there are infinitely many partition values which are odd (respectively, even),  without identifying the values. A major improvement came from Ono  who established that there are infinitely many $n$ such that \( p(an+b) \equiv 0 \pmod{2} \) for any given $a$, $b$ \cite{Ono1}. Later Ono with Boylan showed that the same holds for \( p(an+b) \equiv 1 \pmod{2} \) when $a$ is a power of $2$ \cite{Boylan-Ono}, improving an earlier conditional result in \cite{Ono1}. Cossaboom and Zhu \cite{CZ} recently made the Boylan--Ono  construction explicit. Separately, Radu removed the power-of-$2$ condition on $a$ from Boylan--Ono and showed that no universal congruence of the form $ p(an+b) \equiv 0 \pmod{2}$ is possible \cite[Theorem~1.3~\&~Remark~1.7]{radu}.
By studying powers of \( \Delta(\tau) \coloneqq \eta(\tau)^{24} \) mod~$2$ Bella\"iche and Nicolas \cite[Théorèmes~5~\&~6]{BN} show, for $x > 2$,
\begin{align*}
    & \#\{ n \leq x \mid p(n) \equiv 0 \pmod{2} \} \geq 0.069 \sqrt{x} \log\log(x) \,, \\
    & \#\{ n \leq x \mid p(n) \equiv 1 \pmod{2} \} \geq \frac{0.037\sqrt{x}}{\log^{7/8}{x}} \,.
\end{align*}
These seem to be the best known bounds currently, but they are far from sufficient to even establish a nonzero natural density
\[
     \lim_{x \to \infty} \frac{\#\{ n \leq x \mid \text{$p(n) \equiv 0 \pmod{2}$, \, resp. $ \equiv 1 \pmod{2}$} \}}{x} \overset{?}{>} 0 \,.
\]

Concerning other prime moduli, Ono \cite{Ono3} constructed for any $\ell \geq 5$ an infinite family of essentially arithmetic progressions which each contain infinitely many $n$ such that \(p(n) \equiv 0 \pmod{\ell}\). In further study of 
general congruences like Ramanujan's, for (powers of) the partition function, Kiming and Olsson \cite{KO} showed that a congruence for \( p(n) \) of the form \( p(\ell n + a) \equiv 0 \pmod{\ell} \), \( \ell \geq 5 \) prime, requires~\( 24 a \equiv 1 \pmod{\ell} \).  Consider the sequence \( (\delta_\ell)_\ell \), \( \ell \geq 5 \) prime, where \( \delta_\ell \) is defined by requiring \( 0 < \delta_\ell < \ell \) and \( 24 \delta_\ell \equiv 1 \pmod{\ell} \).  More conceptually \( \delta_\ell \) is defined such that \( p(\delta_\ell) \) is the starting coefficient of~\( P(q) \) under the formal action of the Hecke operator \( U_\ell \) (see \cref{uelldefsec}): that is, \( U_\ell(P(q)) = p(\delta_\ell) q^{\mu_\ell} + \text{higher-order terms} \), with~\( \mu_\ell \) called the \emph{order at infinity} of~$U_\ell P$. The first few terms of this are therefore given by \( \delta_5 = 4 \), \(\delta_7 = 5 \),~\(\delta_{11} = 6 \).  By Kiming and Olssons's work \cite{KO}, this sequence gives the first term in the arithmetic progression \( \ell n + \delta_\ell \) supporting these potential congruences, is hence a natural subsequence to consider for parity and congruence conditions.  For example, Griffin and Ono \cite{GO} showed that for \( \ell \geq 5 \), prime, the sequence \( p(\ell n + \delta_\ell) \) attains infinitely many odd values, the first one of which occurs for \( n < (\ell^2 - 1)/24 \).

Studying values of the partition function \( p(\delta_\ell) \) in the sequence $\delta_\ell$ is challenging for similar reasons as above: the tools to understand negative or half-integral weight modular forms and their images under the shifting-operator \( U_\ell \) are not yet so well developed. However, in the case of positive integral weight modular forms, such as \( \eta^r \) with \( r \) even, {one can study their Fourier coefficients using the theory of Hecke operators and Galois representations, and one can obtain quantitative density results via the Chebotarev Density Theorem.}

We briefly recall the basic Galois-theory method for understanding Fourier coefficients of \emph{eigenforms}. Fix a prime $p$. Starting with a modular eigenform $f = \sum_n a_n(f) q^n$ of some weight $k$ and level $N$ with integer coefficients, a construction of Deligne and Deligne-Serre \cite{Deligne, DeligneSerre} attaches to $f$ a Galois representation 
$\rho_{f}: \Gal(\overline\QQ/\QQ) \to \GL_2(\QQ_p)$
that is unramified at primes $\ell$ not dividing $Np$ and for each such $\ell$ satisfies $\tr \rho_{f}(\Frob_\ell) = a_\ell(f)$, where $\Frob_\ell$ is any Frobenius-at-$\ell$ element in the Galois group. Reducing any integral $\ZZ_p$-lattice in $\rho_{f, p}$ modulo $p$ and semisimplifying if necessary gives us a residual representation depending only on $f$ (and on $p$)
$$\overline\rho_{f}: \Gal(\overline\QQ/\QQ)\to \GL_2(\FF_p)$$
still unramified at $\ell\nmid Np$ and still with $\tr \overline\rho_{f}(\Frob_\ell) \equiv a_\ell(f) \mod{p}$. Since $\overline\rho_f$ has finite image, it factors through a finite Galois group $G_f$. The Chebotarev Density Theorem tells us that Frobenius elements equidistribute in $G_f$, so that, for example, finding the density of primes $\ell$ with $a_\ell(f) \equiv 0 \mod{p}$ is tantamount to counting the proportion of trace-zero matrices in the image of $\overline\rho_f$.

In this paper we use Galois methods developed by Bellaïche to study densities of mod-$p$ modular forms that are not necessarily eigenforms. More precisely, we study the parity of  coefficients of \( \eta^r \) in a subsequence~\( \delta_{\ell,r} \), analogous to the subsequence \( \delta_{\ell}
=
\delta_{\ell,-1} \) for the partition function, to gain some insight into ideas and techniques which might be utilized for studying partition parity later.  
Write \( m_r \coloneqq 24 / \gcd(24, r) \) and $b_r \coloneqq r/\gcd(24,r)$ and consider \( P_r(q) \coloneqq \eta^r(m_r \tau) \), where we have scaled $\tau$ by $m_r$ to make the power series integral. Note that the nonvanishing coefficients of $P_r(q)$ are supported on the powers of $q$ in the arithmetic progression $\{n m_r + b_r\}_n$. 
We therefore set $p_r(n)$ be the coefficient of $q^{n m_r + b_r}$ in~$P_r$. Note that $p_r(n)$, $n m_r + b_r$, and $P_r$, respectively,  specialize to $p(n)$, $n$, and~$P$,
respectively, for $r = -1$.

We define the subsequence~\( \delta_{\ell,r} \) (associated to \( \eta^r \)) as follows.  Set
\[
    \delta_{\ell,r} \coloneqq \frac{1}{m_r} \big( \ell\mu_{\ell,r} - b_r \big),
\]
where \( \mu_{\ell,r} \) is the solution to the equation
\[
    \ell \mu_{\ell,r} \equiv b_r \pmod{m_r} \,, \qquad \text{with} \qquad \frac{b_r}{\ell} \leq \mu_{\ell,r}  < \frac{b_r}{\ell} + m_r  \,.
\]
As above, \( p_r(\delta_{\ell,r}) \) is the first formal coefficient of 
$P_r$
under the action of the \( U_\ell \) operator. Thus~\( p_r(\delta_{\ell,r}) \) tests whether the formal order of infinity matches the actual order at infinity. One then readily sees that~\(\delta_{\ell,r}\) extrapolates $\delta_\ell$ in a natural way. Indeed, $p_{-1}(n) = p(n)$ by construction and the congruence condition on $\mu_{\ell, -1}$ shows that 
$$24\cdot\delta_{\ell, -1} = \ell \mu_{\ell, -1} \equiv 1 \pmod{\ell}.$$
\noindent
Therefore $\delta_{\ell, -1} \equiv 24^{-1} \pmod{\ell}.$ The second condition on $\mu_{\ell, -1}$ ensures that $\delta_{\ell, -1}$ is the minimal positive solution to this congruence in the integers. Thus $\delta_{\ell, -1} = \delta_{\ell}$ for all primes $\ell \neq 2,3.$

In the subsequence \(p_r(\delta_{\ell,r})\) we indeed find biases in the parity for certain $r$.  We consider the associated density
\begin{equation}\label{eqn:density_subsequence}
	\del(r) = \lim_{x \rightarrow \infty} \frac{\#\{\text{$\ell \leq x$ prime} \mid p_r(\delta_{\ell,r}) \equiv 1 \pmod 2\}}{\pi(x)} \,,
\end{equation}
where $\pi(n)$ is the standard prime counting function. Our first result justifies this definition. 

\begin{thma}\label{thm:densityyes}
For every $r \geq 1$, the limit defining $\del(r)$ exists. Moreover, $\del(r)$ is a dyadic rational\footnote{A \emph{dyadic rational} is a fraction of the from $a/{2^k}$ for $a, k \in \ZZ.$}.
\end{thma}

Specializing here again to $r=-1$ gives rise to the density related to the partition function
\begin{equation*}
    \del(-1) = \lim_{x \rightarrow \infty} \frac{\#\{\text{$\ell \leq x$ prime} \mid p(\delta_{\ell}) \equiv 1 \pmod 2\}}{\pi(x)}\,.
\end{equation*}

For a visual representation of this sequence in form of a random walk (\cref{fig:deltawalk}), as well as further information and experiments in this direction, see \cref{complements}.

\begin{remark}\label{1} 
More conceptually, consider the following setup: let $f = \sum_{n} a_n(f) q^n$ be a nonzero power series in~$\ZZ\llb q \rrb$, 
supported on on powers of $q$ that are in an arithmetic progression. In other words, suppose there exist $b$ and $m$, with $\gcd(b, m) = 1$, so that $a_b(f) \neq 0$ and $f \in q^b \ZZ\llb q^m \rrb$, so if $n \not\equiv b \cmod{m}$ then~$a_n(f) = 0$. For a prime $\ell$ not dividing $m$ consider $U_\ell\,f = \sum_n a_{\ell n}(f) q^n$, and let $\nu\coloneqq\nu(\ell, m, b)$ be the minimal integer at least $b$ and congruent to $b$ modulo $m$ that is divisible by $\ell$; write $\nu = \ell \mu$ for~$\mu = \mu(\ell, m, b).$ 
Then we can consider $\mu$ to be the  \emph{formal order at $\infty$ of $U_\ell (f)$}, and wonder whether the coefficient $a_\nu(f)$ of $q^{\mu}$ in $U_\ell (f)$ is divisible by a prime~$p$: that is, whether the true order of $U_\ell\,f$ modulo~$p$ is in fact greater than $\mu$. To this end, define 
\begin{align*}\label{gen} \mathcal D_p(f, m)
&\coloneqq  \mbox{density}\Big(\{ \text{$\ell$ prime} \mid [q^{\mu(\ell, m, b)}] (U_\ell f) \not\equiv 0 \pmod{p}\}\Big)\\
&= \mbox{density}\Big(\{ \text{$\ell$ prime} \mid a_{\nu(\ell, m, b)}(f)  \not\equiv 0 \pmod{p}\}\Big).
\end{align*}

Then $\del(r) = \mathcal D_2(P_r, m_r) = \mathcal D_2\big(\eta^r(q^{m_r}), m_r\big).$
In other words, for $P_r = \eta^r(q^{m_r})$ we are computing the density of primes $\ell$ for which the order at infinity of $U_\ell(P_r)$ modulo $2$ is minimal. See \cref{remark7.4} for a sketch of proof that this density is always well defined. 
\end{remark}

In this paper we show that $P_r$ modulo $2$ is always a modular form modulo $2$ of level $9$. If $3 \mid r$, then in fact $P_r$ modulo $2$ is a mod-$2$ modular form of level $1$ (see \cref{linktodelta}),
which have been extensively studied by Serre, Nicolas, and Bellaïche \cite{NS1, NS2, Brep, Bmem}. 
We both use these results and partially extend them to level $9$. Our first result is a complete criterion for the vanishing of $D(r).$ 

\begin{thma}\label{thm:zero-density}
We have $D(r) = 0$ if and only if at least one of the following holds: 
\begin{enumerate}[ref=\thethma(\roman*), topsep = -3pt] 
\item \label[thma]{thm:zero:mult32} $r$ is a divisor or a multiple of $32$, or 
\item\label[thma]{thm:zero:div48}\label[thma]{thm:zero:mult48} $r$ is a divisor or a multiple of $48$.
\end{enumerate}
\end{thma}

Our next and penultimate result gives upper bounds for $D(r)$. 
\pagebreak
\begin{thma}\label{thm:1/2}
  The following results hold.
  \begin{enumerate}[ref=\thethma(\roman*), topsep = -5pt, itemsep = 2pt, topsep = -3pt, listparindent = 0pt, parsep = 2pt]
  \item\label[thma]{thm:1/2:lt1} We always have $\del(n) < 1$.
  
\item\label[thma]{1/2:1/2} For any $n$ we have $\del(2n) < \frac{1}{2}.$ 
  \item\label[thma]{1/4} For any $n \neq 9, 15, 18, 30$ we have $\del(4n) < \frac{1}{4}.$
  \end{enumerate}
  \end{thma}

Conceptually, \cref{thm:1/2:lt1} means that the order of infinity of \( U_\ell \) applied to any power of \( \eta \), considered modulo 2, fails to be maximal in set of primes \( \ell \) of positive natural density. \cref{1/2:1/2} shows that, when applied to eta powers of the form $P_{2n}$,
the order of infinity of \( U_\ell \) fails to be maximal for more than half the primes \( \ell \).  Finally, \cref{1/4} shows the failure of the order of infinity of \( U_\ell \) to be maximal is even greater (i.e. for more than three-quarters of the primes \( \ell \)) for eta powers of the form~$P_{4n}$.

In practice we see computationally that $D(8n)$ is always less than or equal to $\frac{1}{8}$, but we expect $\frac{1}{8}$ to be attained infinitely often. Similarly, $\del(16n) \leq \frac{1}{16}$, with equality again expected to be attained infinitely often. (See \cref{expectations} for an explanation of both expectations.) On the other side, we observe computationally (for $r < 10^5$: see graph of $D(r)$ values below) that $D(r)$ is always nonstrictly less than \( \frac{5}{8} \), with \( D(r) = \frac{5}{8} \) attained for \( r = 21, 23, 39, 57, 63, 105 \) (and conjecturally for no other $r$).

\bigskip
\includegraphics[width=0.9\textwidth]{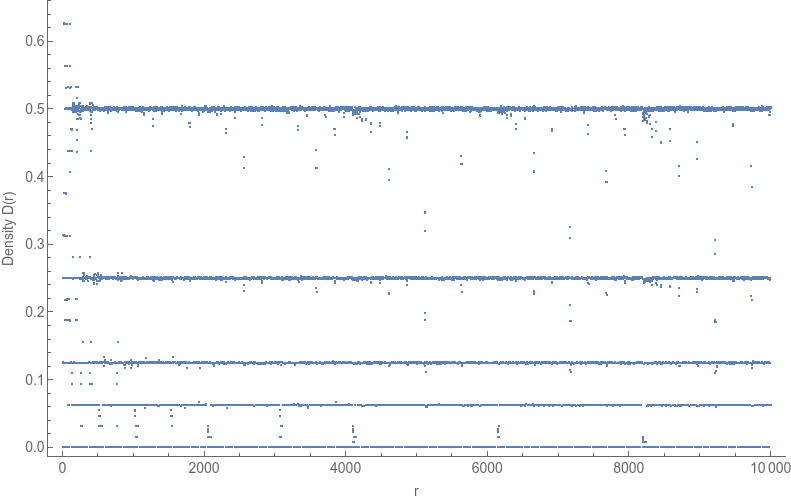}
\bigskip

Our last result concerns powers of eta with unusually small density. For eta powers congruent to mod-$2$ modular forms whose prime Fourier coefficients depend on Frobenius elements in \emph{dihedral} extensions of~$\QQ$, the density $\del(r)$ is typically small, and explicitly computable. 

\pagebreak
\begin{thma}\label{thm:main-thm}
    The following computations of the density \( \del(r) \) hold.
    \begin{itemize}[leftmargin=2em]
     \item {\bf \itshape \( \Q(\sqrt{-2}) \)-dihedral case:}  Consider the sequence \( z_n = \frac{1}{3}(2 \cdot 4^n + 1),\) 
     so $z_n = (3, 11, 43, 171,\ldots)$.  Then \vskip-1\baselineskip
      \begin{align*}
        \del(3 \cdot z_n) = \del(6 \cdot z_n) &= \begin{cases}
            2^{-n}  & \text{if $n\geq2$,}  \\
            \frac{1}{4} & \text{if $n=1;$}
        \end{cases} \\
        \del(12 \cdot z_n) = \del(24 \cdot z_n) & = 2^{-(n+1)} \text{ for $n \geq 1$} \,, \\
        \del(3 \cdot 3z_n) = \del(6 \cdot 3 z_n) &= 3 \cdot 2^{-(n+2)} \text{ for $n \geq 1$} \,, \\
        \del(12 \cdot 3z_n) = \del(24 \cdot 3 z_n) &=  2^{-(n+2)} \text{ for $n \geq 1$} \,,
    \intertext{\item  {\bf \itshape \( \Q(i) \)-dihedral case:}  Consider the sequence \( w_n = 4^n + 1,\) 
    so $w_n=(5, 17, 65, 257,\ldots)$. Then
    }
        \del(3 \cdot w_n) & = \begin{cases}
            3 \cdot 2^{-(n+1)} & \text{if $n \geq 2$,}   \\
            \frac{1}{4} & \text{if $n = 1;$}
            \end{cases} \\
        \del(6 \cdot w_n) = \del(12 \cdot w_n) = \del(24 \cdot w_n) & = 2^{-(n+1)} \text{ for $n \geq 1$}. 
    \end{align*}
    \end{itemize}
\end{thma}
The labels \( \Q(\sqrt{-2}) \)-dihedral and \( \Q(i) \)-dihedral refer to cases in the classification of modular forms modulo~2, which we review in \cref{reviewNSB}.

\begin{remark}\label{thm:abelian-thm}
\emph{Abelian} mod-2 modular forms, like dihedral mod-2 modular forms, will also have explicitly computable density. However, it appears likely that one catches only finitely many eta-powers this way: see \cref{linktodelta} and \cite[Remark \S 3.2]{Bmem}. 
Using \cite[Table]{Bmem} we can compute that 
    \begin{align*}
        \del(3 \cdot 7) &= \del(3 \cdot 19) = \del(3 \cdot 21) = \frac{5}{8}  \,, \\
        \del(6 \cdot 7) &= \frac{3}{8} \,, \qquad
        \del(6 \cdot 19) = \del(6 \cdot 21) = \frac{1}{4}  \,, \\
        \del(12 \cdot 7) &= \del(12 \cdot 19) = \del(12 \cdot 21) = \frac{1}{8} \,, \\
        \del(24 \cdot 7) &= \del(24 \cdot 19) = \del(24 \cdot 21) = \frac{1}{8} \,.
    \end{align*}
Similarly, 
one can compute that for $r$ of the form $an$ for $a \in \{1, 2, 4, 8\}$ and $n \in \{5, 7, 13\}$ --- in other words, for $r$ in  $\{5,7,10,13,14,20,26,28,40,52,56,104\}$ --- we have 
\(
D(r) = \frac{1}{8}.
\)
See \cref{proof:rk:abelian} for details.
\end{remark}

{
\begin{table*}[ht]
\small
\renewcommand{\frac}[2]{#1/#2}
\setlength{\tabcolsep}{1em}
\renewcommand{\arraystretch}{1.5}
\begin{varwidth}{0.16\textwidth}
\begin{tabular}{cC|}\label{coloredtable}
$r$ & \del(r)  \\ \hline
\proven 1  &  0  \\
\proven 2  &  0  \\
 \proven 3  &  0  \\
 \proven 4  &  0  \\
\proven 5  &  \frac{1}{8}  \\
 \proven 6  &  0  \\
\proven 7  &  \frac{1}{8}  \\
 \proven 8  &  0  \\
 \proven 9  &  \frac{1}{4}  \\
\proven 10  &  \frac{1}{8}  \\
 11  &  \frac{5}{16}  \\
 \proven 12  &  0  \\
 \proven 13  &  \frac{1}{8}  \\
\proven 14  &  \frac{1}{8}  \\
 \proven 15  &  \frac{1}{4}  \\
 \proven 16  &  0  \\
 17  &  \frac{5}{16}  \\
 \proven 18  &  \frac{1}{4}  \\
 19  &  \frac{5}{16}  \\
\proven 20  &  \frac{1}{8}  \\
\proven 21  &  \frac{5}{8}  \\
 22  &  \frac{5}{16}  \\
 \end{tabular}
 \end{varwidth}%
\begin{varwidth}{0.16\textwidth}
\begin{tabular}{cC|}
$r$ & \del(r)  \\ \hline
 23  &  \frac{5}{8}  \\
 \proven 24  &  0  \\
 25  &  \frac{5}{16}  \\
\proven 26  &  \frac{1}{8}  \\
 \proven 27  &  \frac{3}{8}  \\
\proven 28  &  \frac{1}{8}  \\
 29  &  \frac{5}{16}  \\
\proven 30  &  \frac{1}{4}  \\
 31  &  \frac{5}{16}  \\
 \proven 32  &  0  \\
 \proven 33  &  \frac{1}{4}  \\
 34  &  \frac{3}{16}  \\
 35  &  \frac{7}{32}  \\
 \proven 36  &  \frac{1}{4}  \\
 \modconj 37  &  \frac{1}{2}  \\
 38  &  \frac{3}{16}  \\
 39  &  \frac{5}{8}  \\
 \proven 40  &  \frac{1}{8}  \\
 41  &  \frac{9}{16}  \\
 \proven 42  &  \frac{3}{8}  \\
 43  &  \frac{9}{16}  \\
 44  &  \frac{3}{16}  \\
 \end{tabular}
 \end{varwidth}%
\begin{varwidth}{0.16\textwidth}
\begin{tabular}{cC|}
$r$ & 
\del(r)  \\ \hline
 \modconja 45  &  \frac{1}{2}  \\
 46  &  \frac{5}{16}  \\
 47  &  \frac{17}{32}  \\
 \proven 48  &  0  \\
 49  &  \frac{5}{16}  \\
 50  &  \frac{5}{16}  \\
 \proven 51  &  \frac{3}{8}  \\
\proven 52  &  \frac{1}{8}  \\
\modconj 53  &  \frac{1}{2}  \\
\proven 54  &  \frac{3}{8}  \\
\modconj 55  &  \frac{1}{2}  \\
\proven 56  &  \frac{1}{8}  \\
 \proven 57  &  \frac{5}{8}  \\
 58  &  \frac{5}{16}  \\
 59  &  \frac{9}{16}  \\
 \proven 60  &  \frac{1}{4}  \\
 61  &  \frac{5}{16}  \\
 62  &  \frac{5}{16}  \\
 \proven 63  &  \frac{5}{8}  \\
 \proven 64  &  0  \\
 65  &  \frac{7}{32}  \\
\proven 66  &  \frac{1}{4}  \\
 \end{tabular}
 \end{varwidth}%
\begin{varwidth}{0.16\textwidth}
\begin{tabular}{cC|}
$r$ &\del(r)  \\ \hline
 67  &  \frac{7}{32}  \\
 68  &  \frac{3}{16}  \\
\modconja 69  &  \frac{1}{2}  \\
 70  &  \frac{7}{32}  \\
 71  &  \frac{17}{32}  \\
 \proven 72  &  \frac{1}{4}  \\
 73  &  \frac{7}{16}  \\
 74  &  \frac{3}{16}  \\
\modconja 75  &  \frac{1}{2}  \\
 76  &  \frac{1}{16}  \\
\modconj 77  &  \frac{1}{2}  \\
\modconja 78  &  \frac{1}{4}  \\
\modconj 79  &  \frac{1}{2}  \\
 80  &  \frac{1}{8}  \\
 \modconja 81  &  \frac{1}{2}  \\
\modconj 82  &  \frac{1}{4}  \\
\modconj 83  &  \frac{1}{2}  \\
 \proven 84  &  \frac{1}{8}  \\
\modconj 85  &  \frac{1}{2}  \\
\modconj 86  &  \frac{1}{4}  \\
\modconja 87  &  \frac{1}{2}  \\
 88  &  \frac{1}{16}  \\
 \end{tabular}
 \end{varwidth}%
\begin{varwidth}{0.16\textwidth}
\begin{tabular}{cC|}
$r$ &\del(r)  \\ \hline
\modconj 89  &  \frac{1}{2}  \\
 \modconja 90  &  \frac{1}{4}  \\
\modconj 91  &  \frac{1}{2}  \\
\modconj 92  &  \frac{1}{8}  \\
\modconja 93  &  \frac{1}{2}  \\
\modconj 94  &  \frac{1}{4}  \\
\modconj 95  &  \frac{1}{2}  \\
\proven  96  &  0  \\
 97  &  \frac{7}{32}  \\
 98  &  \frac{3}{16}  \\
 \proven 99  &  \frac{3}{16}  \\
 100  &  \frac{3}{16}  \\
 101  &  \frac{1}{4}  \\
\proven 102  &  \frac{1}{8}  \\
 103  &  \frac{13}{32}  \\
\proven 104  &  \frac{1}{8}  \\
 105  &  \frac{5}{8}  \\
 106  &  \frac{5}{16}  \\
 107  &  \frac{9}{16}  \\
 \proven 108  &  \frac{1}{8}  \\
 109  &  \frac{17}{32}  \\
 110  &  \frac{5}{16}  \\
 \end{tabular}
 \end{varwidth}%
\begin{varwidth}{0.16\textwidth}
\begin{tabular}{cC}
$r$ & \del(r)  \\ \hline
 111  &  \frac{9}{16}  \\
 112  &  \frac{1}{8}  \\
 113  &  \frac{7}{16}  \\
 \proven 114  &  \frac{1}{4}  \\
 115  &  \frac{17}{32}  \\
\modconj 116  &  \frac{1}{8}  \\
\modconja 117  &  \frac{1}{2}  \\
\modconj 118  &  \frac{1}{4}  \\
\modconj 119  &  \frac{1}{2}  \\
 \proven 120  &  \frac{1}{4}  \\
 121  &  \frac{15}{32}  \\
 122  &  \frac{1}{8}  \\
\modconja 123  &  \frac{1}{2}  \\
\modconj 124  &  \frac{1}{8}  \\
\modconj 125  &  \frac{1}{2}  \\
\proven 126  &  \frac{1}{4}  \\
 127  &  \frac{7}{16}  \\
 \proven 128  &  0  \\
  \proven 129  &  \frac{1}{8}  \\
  130  &  \frac{3}{32}  \\
  131  &  \frac{7}{64}  \\
  \proven 132  &  \frac{1}{8}  \\
       \end{tabular}
       \end{varwidth}%
       \medskip
\renewcommand{\tablename}{}
\caption*{
Expected densities \( \del(r) \) 
for \( \eta^r(m_r \tau) \)
for \( 1 \leq r \leq 132 \).
The entries in \fcolorbox{tblproven}{tblproven}{\llap{\phantom{d}}green} are proved. 
}
\label{tbl:results132}
\end{table*}
}

\cref{thm:main-thm} and \cref{thm:abelian-thm} confirm the numerics illustrated in the table on p.~\pageref{coloredtable}. Here we have converted our computations for $D(r)$ (computed with $2 \times 10^6$ coefficients) into \emph{likely} dyadic rationals.

To prove our results, we relate the density \(\del(r)\) to 
$\bel(f)$, the density of primes with nonzero Fourier coefficients for a mod-2 modular form $f$, first studied by Bella\"iche \cite[\S10]{Bim}. We then use Galois-representation-theoretic techniques to compute $\bel(f)$.

We now describe Bellaïche's Galois-theoretic techniques for understanding Fourier coefficients mod $p$ of modular forms that are not necessarily eigenforms. Unlike the situation in characteristic zero, recall that spaces of mod-$p$ modular forms do not have bases of eigenforms, only bases of \emph{generalized} eigenforms. (In our setting of level $9$ modulo $2$, every form is a generalized eigenform for the eigensystem carried by~$\Delta$; the powers of $\Delta$ form such a basis in level $1$.)
By gluing together Galois representations attached by Deligne to $p$-adic eigenforms of level $N$, one can construct a Galois representation\footnote{More precisely, one obtains a Galois \emph{pseudorepresentation} in the sense of Chenevier \cite{Chenevier} --- see \cref{trace-determinant-identity} for the $p=2$ and $N = 1$ case --- but the distinction between the two is a technicality that we will ignore here in the introduction.} unramfied at primes~$\ell$ not dividing~$Np$, on the big Hecke algebra acting on all mod-$p$ modular forms of level $N$ at once, with the action of Frobenius elements at $\ell$ mapping to the operator $T_\ell$ in the Hecke algebra under trace. If $f$ is a particular mod-$p$ modular form of level $N$, not necessarily an eigenform, then passing to a faithful quotient on the finite Hecke algebra that acts on $f$, we see that the prime Fourier coefficients of $f$ are still determined by the action of Hecke on $f$, which is controlled by Frobenius elements in a finite Galois extension $L_f$ of~$\QQ$ unramified outside $Np$.

The upshot is that for a mod-$p$ modular form $f$, the association of a prime $\ell$ to the $\ell^{\rm th}$ Fourier coefficient~$a_\ell(f)$ is  \emph{$L_f$-frobenian} in the following sense of Serre. 

\begin{definition}[Serre \cite{S0}]\label{def:frobenian-set} 
    Let $K/\Q$ be a finite Galois extension of $\QQ$.
    A map $\varphi: \mathcal P \to S$ from the set of all primes $\mathcal P$ to a set $S$ is called \emph{$K$-frobenian} if for all but finitely many $\ell \in \mathcal P$, the evaluation~$\varphi(\ell)$ depends only on the conjugacy class of $\Frob_\ell$ in $\Gal(K/\Q).$
A map $\varphi: \mathcal P \to S$ is \emph{frobenian} if it is $K$-frobenian for some finite extension $K/\QQ$. A set of primes is \emph{frobenian} if its characteristic function is.\end{definition}
The Chebotarev density theorem then finally allows us to translate questions of density for coefficients of~$f$ into an element-counting problem for the finite group~$\Gal(L_f/\QQ)$. In particular, frobenian sets of primes have well-defined Dirichlet --- and hence natural ---  densities  \cite[Prop.~1.5]{S0}.

The eta-powers that we study in \cref{thm:main-thm} correspond modulo~$2$ to modular forms $f$ whose determining Galois number fields $L_f$ are \emph{dihedral} extensions of $\QQ$ unramified outside $2$, so necessarily contain either~$\QQ(i)$ or $\QQ(\sqrt{-2})$ --- this part of the story already appears in \cite{NS2} and is then limned
explicitly in \cite{Bmem}. 
For such a dihedral form $f$, we use precise analysis of Bellaïche and Serre, partially from~\cite{Bmem} and partially unpublished, to identify the dihedral Galois extension in question. The large cyclic subgroups of dihedral groups now make the rest of the computation possible: on elements of a cyclic group, the connection between Galois elements and Fourier coefficients is particularly simple, governed by (modified) Chebyshev polynomials, whose coefficients can be described through binomial coefficients, whose parity in turn may  be realized by a combinatorial computation involving base-2 digits of the  nilpotence index of the mod-$2$ modular form in question. 

Bellaïche passed away in 2022 without publishing his mod-2 density theorems. Part of our current work thus necessitated presenting proofs of some of his unpublished results: see \cref{item:densityleq1/4} and \cref{bellaiche-density}. We naturally take responsibility for any errors we may have introduced into the arguments.

\bigskip

The paper is structured as follows.  In \cref{sec:elementary}, we start by introducing the elementary definitions and set-up to study the density \( \del(r) \) of an \( \eta \)-power.  In \cref{reviewNSB}, we recall work of Nicolas, Serre, and Bellaïche \cite{NS1, NS2, Brep} on mod-2 modular forms of level 1; in \cref{sec:density}, we further recall the later techniques developed by Bella\"iche in published \cite{Bim, Bmem} and unpublished work  on densities of modular forms modulo 2.  In \cref{sec:combinatorial-lemma} we give 
a proof 
of
a technical combinatorial lemma (\cref{combinatorial-lemma}), 
originally proved by Bella\"iche 
using different methods.  In \cref{sec:lvl9} we extend some results of Nicholas, Serre and Bella\"iche on level-1 mod-2 modular forms to level 9.  Finally, in \cref{finalsection} we related Bellaïche's notion of density 
to the density \( \del(r) \) given above, and prove \cref{thm:densityyes,thm:zero-density,thm:1/2,thm:main-thm}.

\subsection*{Acknowledgements}
We thank Ken Ono for suggesting this investigation and for his guidance during the project. We are grateful to J-P.~Serre for sharing a part of his correspondence with Bellaïche with the third author. Furthermore, we would like to thank Pieter Moree, who motivated the last two authors to compute random walks of partition parities on the hour-long tram ride back to Bonn after the first day of the \href{http://www.mi.uni-koeln.de/Bringmann/Konferenz2024.html}{$q$-series conference in Cologne in March 2024}\footnote{International Conference on
Modular Forms and $q$-Series, 11--15 March 2024.
}. We thank Gessica Alecci and Badri Vishal Pandey for many fruitful discussions on these topics. Finally, we thank the Max Planck Institute for Mathematics in Bonn for hospitality and excellent working conditions. 

\section{Reformulating the Problem}
\label{sec:elementary}

In this section, we set up the elementary definitions and notation encompassing our notion of density for $\eta$-powers.  In particular we define the subsequence \( \delta_{\ell,r} \) and explain its conceptual origins.  We formally introduce  the  notion of density \( \del(r) \) of the \( \eta \)-power \( P_r(q) = \eta^r(q^{m_r}) \), which we wish to compute, and indicate a computational approach (\cref{dplemma}). Finally, we establish a mod-$2$ congruence between powers of eta and modular forms of level dividing $9$ (\cref{linktodelta}).  

\medskip

\subsection{Our notion of density for powers of \texorpdfstring{$\eta$}{eta}}\label{uelldefsec} 
We define as before 
$m_r = \frac{24}{\gcd(24,r)}$ and $b_r = \frac{r}{\gcd(24, r)}$ and consider the function 
$$P_r(q) = \eta^r(m_r\tau) = \sum_{n = 0}^{\infty} 
p_r(n)q^{nm_r + b_r} = 
\sum_{n \geq b_r} 
c(n)q^n,$$
where
$$ c(n) \coloneqq \begin{cases}
p_r(m) & 
\text{$n = m_rm +b_r$ for some $m \geq 0$,}\\

0 & \text{otherwise.} 
\end{cases}$$
The rescaling \( m_r \tau \) gives an integral power series; we are interested in studying the parity of the coefficients~\( p_r(n) \) as \( n \) varies in a special subsequence that we now describe.

Let $f= \sum_{n \geq n_0} a_nq^n$ be a nonzero modular form of some level and integral weight $k$ with integer coefficients; here $n_0$ is the order of~$f$ at $\infty$: that is, we insist that $a(n_0) \neq 0$.
We briefly recall the definition of action of the prime Hecke operator $T_\ell$ on $f$: 
\begin{equation}\label{def:hecke}
    T_\ell(f) \coloneqq \sum_{n=0}^{\infty} \left(a(pn) + p^{k-1}a\left(\frac{n}{p}\right)\right)q^n,
\end{equation}
where we set $a(n/p) = 0$ if $p\nmid n.$ In particular, it follows that $a_1(T_{\ell}(f)) = a_{\ell}(f).$ 

Generally \( T_\ell \) can be decomposed into two pieces which effectively capture the two different rescalings present in the Hecke operator. 

We define 
\begin{align}\label{uelldef}\begin{split}
	U_\ell(f) &\coloneqq \ell^{\frac{k}{2}-1} \sum_{b \pmod{\ell}} f\big|_k \begin{pmatrix}
1 & b \\ 0 & \ell
\end{pmatrix}\\
				&= \frac{1}{\ell} \sum_{b \pmod{\ell}} f\left(\frac{\tau + b}{\ell}\right) = \frac{1}{\ell} \sum_{n \geq n_0} a_n \sum_{b \pmod{\ell}} e^{2 \pi i\left(\frac{\tau + b}{\ell}\right)n}\\
				&=\frac{1}{\ell} \sum_{n \geq n_0} a_n e^{\frac{2\pi i n \tau }{\ell}} \underbrace{\sum_{b \pmod{\ell}} e^{\frac{2\pi i b n}{\ell}}}_{\mathclap{\qquad\qquad\text{$ = \ell$, if $ \ell \mid n $, and 0 otherwise}}}\\
				& = \sum_{\substack{n \geq n_0 \\ n = r\ell}} a(n) q^{\frac{n}{\ell}} = \sum_{r \geq \frac{n_0}{\ell}} a(r\ell) q^r
\end{split}
\end{align}
We also have the operator $V_\ell$ acting on $f$ by \[
V_\ell(f) \coloneqq f(\ell \tau) = \sum_{n \geq n_0} a(n) q^{n\ell}
\]
The aforementioned decomposition is then 
\[
T_\ell(f) = U_\ell(f) + \ell^{k-1} V_\ell(f).
\]

Now we apply $U_\ell$ on $P_r(q)$ explicitly:
\begin{align*}
	U_\ell(P_r(q)) &= \sum_{n \geq \frac{rm_r}{24\ell}} c(n\ell)q^n = \sum_{\substack{n \geq \frac{rm_r}{24\ell} \\ n\ell \equiv \frac{rm_r}{24} \pmod {m_r}}} p_r\left(\frac{n\ell-\frac{rm_r}{24}}{m_r}\right) q^{n}\\
						&= p_r\biggl(\,\underbrace{\frac{\mu_{\ell,r}\ell-\frac{rm_r}{24}}{m_r}}_{\coloneqq \delta_{\ell,r}}\,\biggr) q^{\mu_{\ell,r}} +\cdots,
\end{align*}
where $\mu_{\ell,r}$ is defined as the solution of the equation
\[
\ell \mu_{\ell,r} \equiv \frac{rm_r}{24} \pmod{m_r} \quad\mbox{with}\quad \frac{rm_r}{24\ell} \leq \mu_{\ell,r} < \frac{rm_r}{24\ell} + m_r.
\]
We now define the following density.
\begin{definition}\label{def:subdensity}
The density \( \del(r) \) of nonzero values in the sequence \( p_r(\delta_{\ell,r}) \) is defined by
\[
	\del(r) = \lim_{x \rightarrow \infty} \frac{\#\{\text{$\ell \leq x$ prime} \mid p_r(\delta_{\ell,r}) \equiv 1 \pmod 2\}}{\pi(x)} \,,
\]
where \( \pi(n) \) is the prime counting function. We discard the primes $\ell=2,3$ from the definition due to congruence obstructions.
\end{definition}

\vspace{-0.8em}

One of our main results (\cref{thm:densityyes}) is that the limit in \cref{def:subdensity} always exists.
\begin{example}
    We give some examples to clarify these definitions.

    \paragraph{\bf Case $r=18$}
    We directly have $m_{18} = 4$.  If we consider \( \ell = 5 \), we have that \( \mu_{5,18} \) is defined by satisfying $ 5 \mu_{5,18} \equiv 3 \pmod{4} $, with \( \frac{20}{24\cdot 5} \leq \mu_{5,18} < \frac{20}{24\cdot5} + 4 \).  Thus, $\mu_{5,18} = 3$, which gives us
\[
U_5(P_{18}(q)) = \sum_{\substack{n \geq \frac{3}{5} \\ 5n \equiv 3 \pmod{4}}} p_{18}\left(\frac{5n - 3}{4}\right) q^n = p_{18}(3)q^3 + \cdots
\]

    \paragraph{\bf Case $r=120$} We get $m_{120} = 1$.  If we consider $\ell = 7$, then \( \mu_{7,120} \) satisfies $ 7 \mu_{7,120} \equiv 5 \pmod {1}$ and by the second condition we see that $\mu_{7,120} = 1$.  Thus
    $$U_7(P_{120}(q)) = \sum_{n \geq \frac{5}{7}} p_{120}(7n-5)q^n = p_{120}(2)q + \cdots$$
    Likewise, if we look at $\ell = 5$, the conditions on \( \mu_{5,120}  \) require \( 5  \mu_{5,120} \equiv 5 \pmod {1} \), so we find $\mu_{5,120} = 1$. Thus,
    $$U_5(P_{120}(q)) = \sum_{n \geq 1} p_{120}(5n-5)q^n = p_{120}(0)q + \cdots$$

As a last example we consider $\ell = 13$. By the same calculations as before $\mu_{13,120} = 1$ and we find 
$$U_{13}(P_{120}(q)) = \sum_{n \geq 1} p_{120}(13n-5)q^n = p_{120}(8)q + \cdots$$
\end{example}
\subsection{Eta powers as mod-\texorpdfstring{$2$}{2} modular forms}
Eventually we will relate the density $\del(r)$ above to the notion of density $\bel(f)$ of a mod-2 modular form~$f$ as defined by Bellaïche in \cite{Bim}; see \cref{def:beldensity} below for details. We begin by identifying eta powers modulo $2$ with mod-$2$ modular forms of level $1$  and level $9$.  
Recall that $P_r(q) = \eta^r(q^{m_r})$, where $m_r = \frac{24}{\gcd(24, r)}$ and $b_r = \frac{r}{\gcd(24, r)}$. 

Let \[ C = q - 8q^4 + 20q^7 - 70q^{13} + 64q^{16} + 56q^{19} + O(q^{21}) \in S_4\big(\Gamma_0(9)\big) \] be the unique normalized cusp form of weight 4 and level 9.  It is well know that \( C = \eta(3\tau)^8 \): for example \cite[Proposition 3.1.1, proof]{L} shows that \( \eta(3\tau)^8 \) is in  \(S_4\big(\Gamma_0(9)\big) \) as it transforms  appropriately; and  dimension formulas (see \cite[Theorem 1]{CO}, for example) give \( \dim S_4\big(\Gamma_0(9)\big) = 1 \).
In particular, 
\begin{equation}\label{c3delta}
C(q)^3 = \Delta(q^3).
\end{equation}

\begin{lemma}\label{lem:Casdelta} We have $C(q) \equiv \Delta(q) + \Delta(q^9) \pmod{2}.$
\end{lemma}

\begin{proof}
 First we note that, by squaring modulo $2$,
        \[
            \eta(3\tau)^8 = q \prod_{n=1}^{\infty}\left(1-q^{3n}\right)^8 \equiv q \prod_{n=1}^{\infty} \left(1-q^{24n}\right) = P_1(q) \pmod{2} \,.
        \]
        Thus, it is enough to show that $P_1(q) \equiv \Delta(q) + \Delta(q^9) \pmod{2} $.
        For this, the Pentagonal Number Theorem \cite[Corollary 1.7]{Andrews} gives the following equality
        \begin{equation}\label{eqn:penta}
            \prod_{n=1}^{\infty} (1-q^n) = 1 + \sum_{k=1}^{\infty} (-1)^k\left(q^{k(3k+1)/2} + q^{k(3k-1)/2}\right) \,,
        \end{equation}
        and the Jacobi triple product identity \cite[Theorem 2.8]{Andrews} leads to the following well-known $q$-expansion for~\( \Delta \) modulo 2:
        \begin{equation}\label{eqn:jacobi_triple}
            \Delta(q) \equiv q\sum_{n=0}^{\infty} q^{4n(n+1)} = \sum_{n=0}^{\infty} q^{(2n+1)^2} \pmod{2}.
        \end{equation}
        Therefore by \eqref{eqn:penta} and \eqref{eqn:jacobi_triple} it is enough to show that
        \[
            \sum_{k=1}^{\infty} q^{4k^2+4k} = \sum_{k=1}^{\infty} q^{36k^2+12k} + \sum_{k=1}^{\infty} q^{36k^2-12k} + \sum_{k=1}^{\infty} q^{36k^2-36k+8}.
        \]
        Restricting the sum on the left hand side to the congruences $k \equiv 0, 1, 2 \pmod{3}$ gives the first, second and third sum, respectively, on the right hand side.  
\end{proof}

From \eqref{eqn:jacobi_triple}, and \cref{lem:Casdelta} we obtain the following $q$-expansions of \( C \) and \( \Delta \) modulo 2:
\begin{alignat}{3}
 \label{deltaeq}   \Delta(q) & \equiv \sum_{\text{$n$ odd}} q^{n^2} && = \,\, q + q^9+ q^{25} + q^{49} + q^{81} + q^{121} + q^{169} + q^{225} + O(q^{289}) \quad \pmod{2}\\
  \label{Ceq}  C(q) & \equiv \sum_{\substack{\text{$n$ odd} \\ 3 \nmid n}}  q^{n^2} && =
  \,\, q + q^{25} + q^{49} + q^{121} + q^{169} + O(q^{289})
  \quad \pmod{2} \,.
\end{alignat}

\begin{prop}\label{linktodelta}
For $r \geq 1$, and with \( b_r = \frac{r}{\gcd(24, r)} \) as before, we have the following congruence of mod-$2$ $q$-expansions $$P_r \equiv \begin{cases} \Delta^{b_r} & \mbox{ if $3 \mid r$, and } \\
C^{b_r} & \mbox{otherwise.}\end{cases} $$
\end{prop}

    \begin{proof}
        Recall the notation \( m_r = \frac{24}{\gcd(24, r)} \) and \( b_r = \frac{r}{\gcd(24, r)} \). Then as before \( P_r = q^{b_r} \prod_{n=1}^\infty (1 - q^{m_r n} )^r \). 

        First, suppose \( 3 \mid r \), so \( r = 3s \), giving \( b_r = \frac{s}{\gcd(8,s)} \) and \( m_r = \frac{8}{\gcd(8,s)} \), so that in particular \( m_r \) is a power of 2.  Then
        \begin{align*}
            P_r  = q^{{s}/{\gcd(s, 8)}} \prod_{n=1}^\infty \big(1 - q^{{8n}/{\gcd(s,8)}} \big)^{3s}  
             = q^{{s}/{\gcd(s, 8)}} \prod_{n=1}^\infty \big(1 - q^{n}\big) ^{{8 \cdot 3s}/{\gcd(s,8)}} 
             = \Delta^{{s}/{\gcd(s, 8)}} 
             = \Delta^{b_r} \,.
        \end{align*}

        Otherwise, suppose \( 3 \nmid r \).  Then \( b_r = \frac{r}{\gcd(r, 8)} \), and \( m_r = \frac{24}{\gcd(r, 8)} \), so 
        \begin{equation*}
        P_r = q^{{r}/{\gcd(r, 8)}} \prod_{n=1}^\infty \big(1 - q^{{3n \cdot 8}/{\gcd(r, 8)}} \big)^r 
         = q^{{r}/{\gcd(r, 8)}} \prod_{n=1}^\infty (1 - q^{3n})^{{8r}/{\gcd(r,8)}} 
         =\big(\eta(3\tau)^8\big)^{{r}/{\gcd(r,8)}} 
        =C^{b_r}. \qedhere
        \end{equation*}
\end{proof}

\begin{lemma}\label{dplemma}
Suppose $f$ is a $q$-series supported on arithmetic progression $b$ modulo $m$. For $c$ with $\gcd(c, m) = 1$, let $1\leq u_c< m$ be the least positive integer congruent to $b c^{-1}$ modulo $m$. Then 
$$\mathcal D_p(f, m) = \sum_{\gcd(c, m) = 1} 
\operatorname{density}\big(\{\ell \equiv c \pmod{m}: a_{u_c \ell}(f) \not\equiv 0 \pmod{p}\}\big).
$$
Furthermore, if $f$ is a modular form, then 
$$\mathcal D_p(f, m) = 
\sum_{\gcd(c, m) = 1} 
\operatorname{density}\big(\{\ell \equiv c \pmod{m}: a_\ell(T_{u_c} f)  \not\equiv 0 \pmod{p}\}\big).
$$

\end{lemma}
\begin{proof}
We are free to ignore the finitely many primes $\ell < b$ and primes $\ell$ dividing $m$. For each other prime $\ell \equiv c \pmod{m}$, the quantity $\ell u_c$ is the smallest multiple of $\ell$ in the congruence class $b$ modulo $m$. Therefore in the notation of \cref{1} $\ell u_c = \nu(\ell, m, b).$ If $f$ is in a  space of modular form, then $T_{u_c}$ is a meaningful operator and $a_{u_c\ell}(f) = a_\ell(T_{u_c}f).$
\end{proof}

\section{Review of modular forms modulo 2 of level \texorpdfstring{$1$}{1}} \label{reviewNSB}

In this section, we review the work of Nicolas, Serre, and Bellaïche \cite{NS1, NS2, Bmem, Brep, Sletter}, and the techniques by which one can study the coefficients of a mod-2 modular form through representations of a Galois group. With the exception of \cref{invariance-of-basis-under-pinning} (which was certainly known to Bellaïche), everything in this section is a review and appears in the published literature.  

With the exception of \cref{lem:structureLevel1}, we shall henceforth write
\begin{align*}
    \Delta(q) & \coloneqq q + q^9+ q^{25} + q^{49} + q^{81} + q^{121} + q^{169} + q^{225} + O(q^{289}) \\
    C(q) & \coloneqq q + q^{25} + q^{49}  + q^{121} + q^{169} +  O(q^{289}) \,,
\end{align*}
to mean the mod-2 reductions of the previously considered modular forms \( \Delta \) and \( C \).

\subsection{The space of mod-\texorpdfstring{$2$}{2} modular forms and their Hecke algebra}
We begin with recalling the work of Nicolas and Serre on modular forms of level 1 modulo $2$ and the big completed Hecke algebra acting on this space.

\subsubsection{The space of forms}  We first establish the structure of modular forms of level 1 modulo 2.
\begin{lemma}\label{lem:structureLevel1}
	Any integral modular form  of level $1$ is congruent modulo $2$ to a polynomial in $\Delta$.
\end{lemma}
\begin{proof}  The proof is given in \cite{SWD}, and follows from the fact that a \( \ZZ \)-basis for modular forms of level 1 for \( \SL_2(\Z) \) is given by $\ZZ[\Delta, E_4] \oplus E_6 \ZZ[\Delta, E_4].$
\end{proof}

Since squaring commutes with the Hecke action, we are free to work with a complement of the subspace of squares; the natural choice is the Hecke-invariant kernel of the $U_2$ operator. 
With this in mind, let $K$ be the subspace of $\mathbb{F}_2\llb q\rrb$ spanned by the odd powers of $\Delta$, which we identify with its mod-$2$ $q$-expansion:
$$K \coloneqq \FF_2\langle \Delta, \Delta^3, \Delta^5, \dots \rangle \,.$$ 

\subsubsection{The grading on the space of forms}
We also define the spaces $K^i$ for $i = 1,3,5,7$ as
\begin{equation} \label{eqn:Kispan}
K^i \coloneqq \FF_2\big\langle \Delta^k \mid k \equiv i \pmod{8} \big\rangle \,.
\end{equation}
Then clearly
\begin{equation}\label{kgrade} K = K^1 \oplus K^3 \oplus K^5 \oplus K^7\,.\end{equation}
We note \cite[(1)]{NS1} that the coefficients of \( \Delta^i \) are supported on \( q^n \) such that \( n \equiv i \pmod{8} \).  This follows from the fact that this is clearly true for $\Delta$, as $\Delta = \sum_{\text{$n$ odd}} q^{n^2}$. Therefore  \eqref{kgrade} defines a \( (\Z/8\Z)^\times \)-grading on \( K \), with $K^i = \{f \in K: a_n(f) \neq 0 \implies n \equiv i \mod{8}\}.$

\subsubsection{The construction of the Hecke algebra}
We now construct the Hecke algebra acting on $K$. Let $K_n$ be the subspace of $K$ of degree up to $n$:
$$K_n \coloneqq \FF_2 \big\langle \Delta, \Delta^3, \dots, \Delta^{2\lfloor{n/2}\rfloor} \big\rangle.$$
Let $A_n$ be the (finite, commutative) $\FF_2$-subalgebra of $\End_{\FF_2}(K_n)$ generated by action of Hecke operators~$T_\ell$ for odd primes~$\ell$.
Since \( K_n \subseteq K_{n+1} \), and the Hecke action is compatible, we obtain surjective maps~\( A_{n+1} \twoheadrightarrow A_{n} \) by restriction \cite[\S4]{NS2}.
Thus the $A_n$ form a projective system; consider the inverse limit
$$ A \coloneqq \varprojlim A_n.$$
Then $A$ is a pro-$2$ $\FF_2$-algebra, the \emph{shallow} or \emph{anemic} (that is, no $U_2$) big Hecke algebra acting on modular forms modulo $2$ of level $1$. Equivalently, $A$ is the closure in $\End_{\FF_2}(K)$ of the algebra generated by the action of the $T_\ell$, under the compact-open topology on $\End_{\FF_2}(K)$ induced by the discrete topology on $K$.

\subsubsection{Duality}\label{subsec:duality}
By \cite[Théorème~5.1]{NS2}, 
the  classical pairing 
\begin{equation}\label{eqn:duality}
\begin{aligned}[t]
    \langle \text{--}, \text{--} \rangle \colon A \times K &\to \mathbb{F}_2 \\
    (T,f) &\mapsto a_1(T f) \,.
\end{aligned}
\end{equation}
is a continuously perfect pairing, inducing $A$-equivariant isomorphisms $$A \simeq K^\vee \coloneqq \Hom(K, \FF_2) \mbox{\qquad and \qquad} K \simeq A^{\vee, {\rm cont}} \coloneqq \Hom_{\rm cont}(A, \FF_2).$$ 

For $f \in K$, write $f^\ast$ for the continuous linear form on $A$ induced by $f$. Namely, $f^\ast(T) = \langle T, f \rangle = a_1(Tf).$

\subsubsection{The structure of the Hecke algebra}
In fact, $A$ is a \emph{local} $\FF_2$-algebra, as was already known to Tate in the 1970s \cite{Tmod2}. This is because the maximal ideal of $A$ are in correspondence with mod-$2$ Hecke eigensystems, and a mod-$2$ Hecke eigensystem other than that of $\Delta$ would force the existence of an $S_3$-extension of $\QQ$ unramified outside $2$. Since no such extensions exist, every form in $K$ is a generalized eigenform for the $\Delta$-eigensystem, and $A$ is local. In particular every $T_\ell$ is in the maximal ideal $\mm$ of $A$, and acts locally nilpotently on $K$.

Nicolas and Serre show further that $A$ a regular local $\FF_2$-algebra of dimension $2$.
\begin{theorem}[Nicolas-Serre  {\cite[Th\'eor\`eme 4.1 and \S4]{NS2}}]\label{nsthm}
    The homomorphism $\FF_2 \llb x, y \rrb \to A$ mapping $x \mapsto T_3$ and $y \mapsto T_5$ is an isomorphism of complete local noetherian $\FF_2$-algebras.
\end{theorem}

Of course there are many other choices of generators of $A$. However, the grading on $A$ provides additional structure, which we describe below.

\subsubsection{The grading on the Hecke algebra}

From the definition of the Hecke operator $T_\ell$ as in \eqref{def:hecke} and the description of $K^i$ from \eqref{eqn:Kispan}, it is easy to see that  \( T_\ell(K^i) \subset K^{\ell i} \).  This induces a grading $(\ZZ/8\ZZ)^\times$-grading on $A$ making $K$ into a graded $A$-module.

Identifying $x$ with $T_3$ and $y$ with $T_5$ as in \cref{nsthm}, define the subalgebra
    \( A^1 \coloneqq \mathbb{F}_2\llb x^2, y^2 \rrb \) and the $A^1$-submodules \( A^3 \coloneqq x A^1 \), \( A^5 \coloneqq y A^1 \), and \( A^7 \coloneqq xy A^1 \).  It is clear that 
    \[
       A = A^1 \oplus A^3 \oplus A^5 \oplus A^7 \,,
    \]
     is a $(\ZZ/8\ZZ)^\times$-graded algebra, and $K$ is a graded $A$-module, meaning: 
  \[
        A^i A^j \subset A^{ij} \,, A^i K^j \subset K^{ij} \,,
     \]
     where the indices are always taken modulo 8.

\subsubsection{Graded parameters on $A$} The grading on $A$ invites us to focus on parameters that respect the grading. These are called a \emph{good system of generators} in \cite{Bmem}, but we use the term \emph{graded parameters}.
\begin{definition}[{\cite[\S2.7, Definition 1]{Bmem}}]\label{def:gradedparams}
    Let \( X \) and \( Y \) be be two elements of \( A \).  We call \( (X, Y) \) a  \emph{pair of graded parameters} if there exists two invertible elements \( a, b \in A \), such that \( X = xa^2 \), and \( Y = yb^2 \).
\end{definition}

A pair $(X, Y)$ of graded parameters topologically generated $A$: $A = \FF_2\llb X, Y \rrb$. Furthermore, they respect the grading, in that $A^1 = \FF_2 \llb X^2, Y^2 \rrb$, $A^3 = X A^1$, $A^5 = Y A^1$, and $A^7 = XY A^1$. 

Of course  \( (X,Y) = (x, y) = (T_3, T_5) \) is a pair of graded parameters, although one can in fact also take \( (X,Y) = (T_p, T_q) \) for any primes \( p \equiv 3 \pmod{8} \) and \( q \equiv 5 \pmod{8} \) \cite[\S7, {before Remarques}]{NS2}.  
For further generalization relying on the image of the Galois pseudorepresentations carried by $A$, see \cref{sec:structureofgaloisgp} below.

\subsubsection{Basis for the space of forms adapted to a pair of graded parameters}
As a consequence of the duality in \eqref{eqn:duality}, any pair of graded parameters $(X, Y)$ determines a \emph{graded basis adapted to $(X, Y)$}. 

\begin{definition}[ {\cite[\S2.7, Definition 2]{Bmem}}]\label{def:adaptedbasis}
  Given graded parameters \( (X,Y) \), a \emph{graded basis adapted to $(X, Y)$} 
  is a basis \( \{ m_{(X,Y)}(a,b) \}_{a\geq0,b \geq 0} \) of \( K \) satisfying
  \begin{enumerate}[topsep = -3pt, ref={\thedefinition(\roman*)}]
		\item $m_{(X,Y)}(0,0) = \Delta$,
		\item  $X \, m_{(X,Y)}(a,b) = m_{(X,Y)}(a-1,b)$ if $a > 0$,
		\item  $Y \, m_{(X,Y)}(a,b) = m_{(X,Y})(a,b-1)$ if $b > 0$,
		\item\label[definition]{def:adaptedbasis:ord}   $a_1(m_{(X,Y)}(a,b)) = 0,$ if $a+b > 0$.
	\end{enumerate}
\end{definition}

\begin{prop}[{Nicolas-Serre \cite[Th\'eor\`eme 6.1]{NS2}, Bella\"iche \cite[Proposition 19, \S2.7]{Bmem}}]\label{prop:adaptedbasis}
  Any pair of graded parameters has a unique adapted graded basis.

  \begin{proof}
      The proof is taken from \cite{NS2}.  Using the isomorphism \( K \cong A^{\vee, \mathrm{cont}}\) it suffices to prove the result for \(  A^{\vee, \mathrm{cont}} \).  Here define \( m_{(X,Y)}(a,b) \) as the linear form dual to $X^aY^b$ on \( A \), sending
      \[
        \sum n_{ij} X^i Y^j \mapsto n_{ab} \,.
      \]
      Properties (i)--(iv) follow directly.  The uniqueness follows by recurrence on \( a+b \).
  \end{proof}
\end{prop}

\begin{lemma}\label{invariance-of-basis-under-pinning} 
    Let \( (X,Y) \) and \( (X,Y') \) be two pairs of graded parameters agreeing in the first component.  Then the following equality holds for all \( a \geq 0 \),
    \[
        m_{(X,Y)}(a,0) = m_{(X,Y')}(a,0) \,.
    \]
    Likewise, for \( (X,Y) \) and \( (X', Y) \) , two pairs of graded parameters agreeing in the second component, the following equality holds for all \( b \geq 0 \),
    \[
    m_{(X,Y)}(0,b) = m_{(X',Y)}(0,b) \,.
    \]

    \begin{proof}
        We treat the first case; the second case is similar.  On \( A = \mathbb{F}_2\llb X,Y \rrb = \mathbb{F}_2\llb X, Y'\rrb \), the basis element \( m_{(X,Y)}(a,0) \) is the linear form sending
        \[
            m_{(X,Y)}(a,0) \colon \sum n_{ij} X^i Y^j \mapsto n_{a0} \,,
        \]
        likewise \( m_{(X,Y')}(a,0) \) is the linear form sending
        \[
            m_{(X,Y')}(a,0) \colon \sum n'_{ij} X^i (Y')^j \mapsto n'_{a0} \,.
        \]
        By \cref{def:gradedparams}, it follows that \( Y' = Y \beta^2 \), for some invertible \( \beta \).  Invertible means the constant term of~\( \beta \) is 1, so~
        \(
    \beta^2 \in 1 + (X^2,Y^2) 
        \).
        We see directly that the coefficient of \( X^i Y^0 \) in 
        \[
            \sum n'_{ij} X^i (Y')^j = \sum n'_{ij} X^i Y^j \beta^{2j}
        \]
        is \( n'_{i0} \).  In particular, this means that as a linear form, \( m_{(X,Y)}(a,0) \) also sends
        \[
            \sum n'_{ij} X^i (Y')^j \mapsto n'_{a0} \,.
        \]
        This shows that \(  m_{(X,Y')}(a,0) =  m_{(X,Y)}(a,0) \).  The second case is similar, so the claim holds.
    \end{proof}
\end{lemma}

It will be convenient to introduce the notation
\begin{equation}\label{eqn:mabdef}
    m(a, b) \coloneqq m_{(T_3, T_5)}(a,b) \,,
\end{equation}
to suppress the graded parameters $(X, Y)$ from the notation in the case $(X, Y)  = (x,y) =  (T_3,T_5)$.

\subsection{Review of pseudorepresentations of dimension \texorpdfstring{$2$}{2} and determinant \texorpdfstring{$1$}{1}}

In the next two sections we recall the work of Bellaïche from \cite{Brep} on the Galois (pseudo)representation carried by $A$.

\begin{definition}[Chenevier \cite{Chenevier}]\label{trace-determinant-identity}
    Let $\Gamma$ be a group and $B$ be a commutative ring. A \emph{pseudorepresentation} (of dimension $2$ with determinant $1$, in the sense of Chenevier) of $\Gamma$ on $B$ is a function $t: \Gamma \to B $
satisfying the following
\begin{enumerate}[ref=\thedefinition(\roman*)]
        \item  $t(1) = 2,$
        \item $t$ is central, i.e. $t(gh) = t(hg)$ for all $g,h \in \Gamma$,
        \item\label[definition]{cond:tr-det} $t(gh) + t(g^{-1}h) = t(g)t(h)$ for all $g,h \in \Gamma$.
    \end{enumerate}
\end{definition}

The function $t$ is designed to mimic the properties of trace of a determinant-$1$ representation. Indeed, if $\rho: \Gamma \to \SL_2(B)$ is a representation, it is easy to show that $t \coloneqq \tr \rho$ is a pseudorepresenation of $\Gamma$ on $B$ as above. The converse is true if $B$ is an algebraically closed field. 
Pseudorepresentations in the sense of Chenevier generalize earlier notions of \emph{pseudocharacters} developed by Wiles, Taylor, and Rouquier to the setting where the characteristic is nonzero and does not exceed the dimension.

The \emph{kernel} of a pseudorepresentaion $t: \Gamma \to B$ is the set 
$$\ker t \coloneqq \{g \in \Gamma, t(gh) = t(g) \mbox{ for all $h \in \Gamma$}\}.$$
One can check that $\ker t$ is a normal subgroup of $\Gamma$ through which $t$ factors, and that the induced map $t: G/\ker t \to B$ has trivial kernel. Furthermore, if $\Gamma$ is a profinite group, $B$ is a profinite ring, and $t: \Gamma \to B$ is continuous, then $\ker t$ is a closed normal subgroup of $\Gamma$.

\subsection{The pseudorepresentation carried by \texorpdfstring{$A$}{A}}\label{pseudo}
Fix an algebraic closure $\overline{\Q}$ of $\Q$. Let $E$ be the maximal pro-$2$ extension of $\QQ$ unramified outside $2$.
This is the union of all finite Galois extensions~$L$ of $\QQ$ that are unramified outside $2$ whose Galois group is a $2$-group. Then $E$ is automatically Galois over $\QQ$; set $G \coloneqq \Gal(E/\QQ)$, a profinite topological group. In particular, $\Gal(L/\QQ)$ for any~$L$ as above is a continuous quotient of $G$.

By gluing together traces of representations attached by a theorem of Deligne to $2$-adic eigenforms of level~$1$ and then reducing modulo $2$, Bellaïche has shown the following. 

\begin{theorem}[Bellaïche, {\cite[Théorème 1]{Brep}}]
    There exists a unique continuous pseudorepresentation $t: G \rightarrow A$ with $t(\Frob_\ell) = T_\ell$ for all odd primes $\ell$.
\end{theorem}

From deformation theory we know that the cotangent space $\mm/\mm^2$ of $A$ has, as a basis $\{t(g_3), t(g_5)\}$ for any $g_3 \in G^3$ and $g_5 \in G^5$ \cite[Lemma 7.76]{Chenevier} or \cite[\S 7.5]{Mthesis}. 
Combined with the grading, we know that $t(g_3) \in (x + \mm^2) \cap A^3$ and $t(g_5) \in (y + \mm^2) \cap A^5$. It follows that $\big(t(g_3), t(g_5)\big)$ is a pair of graded parameters in the sense of \cref{def:gradedparams}~\cite[\S 2.7]{Bmem}. 

In the next section we specialize these to graded parameters coming from \emph{pinnings of the Galois group}.

\subsection{The structure of Galois group of interest}\label{sec:structureofgaloisgp}
We now recall the structure of $G = \Gal(E/\Q)$, as studied by Serre \cite[\S 2]{Bmem} following initial investigations by Markshaitis \cite{M}. 

Let $L^\ast = \QQ(\zeta_8) \subseteq E$. Markshaitis shows that $G^\ast \coloneqq \Gal(E/L^\ast) \subset G$ is the Frattini subgroup\footnote{The Frattini subgroup of a pro-$p$ group $\Gamma$ is the smallest closed normal subgroup of $\Gamma$ generated by commutators and~$p^{\rm th}$ powers, so that the quotient is the largest elementary $p$-group quotient of $\Gamma$.} of~$G$. We identify the quotient $G/G^\ast = \Gal(L^\ast/\QQ) = \Gal(\QQ(\zeta_8)/\QQ)$ with $(\ZZ/8\ZZ)^\times$ in the usual way: $a \in \ZZ/8\ZZ$ acts on any primitive $8^{
\rm th}$ root of unity $\zeta_8$ by raising it to the $a^{\rm th}$ power. Denote the quotient map $Q: G \to (\ZZ/8\ZZ)^\times$. This gives a coset decomposition $$G = G^1 \cup G^3 \cup G^5 \cup G^7,$$ where $G^i = Q^{-1}(i)$, making $t: G \to A$ into a $(\ZZ/8\ZZ)^\times$-graded pseudorepresentation in the sense of~\mbox{\cite[\S 3]{deomed}}: for odd primes $\ell$ the conjugacy class of $\Frob_\ell$ is contained inside $G^\ell$.

Serre proves that $G$ has presentation $\langle g, c \mid c^2 =1 \rangle$ in the category of pro-$2$ groups, where $c$ is a complex conjugation (equivalently, as Serre shows, $c$ is any element of order $2$) and $g$ is any element in $G^3$ or in $G^5$ \cite[\S 1]{Bmem}. A choice $(g, c)$ satisfying these properties will be called a \emph{pinning} of $G$. In the sequel we will focus on pinnings with $g \in G^3$, which we call a \emph{$3$-pinning}. If $e = (g, c)$ is a $3$-pinning of $G$, then it follows from the discussion in the last paragraph of \cref{pseudo} that $\big( t(g), t(cg)\big)$ is a pair of graded parameters for $A$. 
In particular, by \cref{prop:adaptedbasis} we can find an adapted basis for $\big( t(g), t(cg)\big)$, which we denote $\{m_e(a, b)\}_{a \geq 0, b\geq 0}.$

\subsection{The field of determination; abelian and dihedral forms} \label{tf}
For a mod-$2$ modular form $f$ in~$K$, write  $\ann(f)$ for the ideal of $A$ annihilating $f$, and write $A_f \coloneqq A/\ann(f)$, the quotient of $A$ acting faithfully on the Hecke span $K_f \coloneqq Af$ of $f$ in $K$. The perfect pairing in \eqref{eqn:duality} restricts to a perfect pairing $A_f \times K_f \to \FF_2,$ so that both $A_f$ and $K_f$ are of the same finite $\FF_2$-dimension.

Let $t_f: G \stackrel{t}\to A \twoheadrightarrow A_f$ be the composition of $t$ with the quotient $A \twoheadrightarrow A_f$. This is a pseudorepresentation with finite image, so it factors through a finite group $G_f \coloneqq G/\ker t_f$. Let $L_f$ be the fixed field of $\ker t_f$. The upshot is that we obtain a pseudorepresentation $t_f: G_f = \Gal(L_f/\QQ) \to A_f$ of a finite Galois 2-group on a finite local $\FF_2$-algebra, still satisfying $t(\Frob_\ell) = T_\ell$. From \eqref{def:hecke} we have $a_\ell(f) = a_1(T_\ell f)$, so that for any odd prime $\ell$, the $\ell^{\rm th}$ coefficient of $f$ is determined by $\Frob_\ell$ in the finite Galois extension $L_f$ over $\QQ.$ Any such field is sometimes called a \emph{field of determination} of $f$ (see \cite[\S 11]{Bim} and \cite[\S 3]{Bmem}); our $L_f$ defined here is the minimal such field.

In \cite{Bmem} Bellaïche defines and analyses two special families of forms in $K$: \emph{abelian} forms and \emph{dihedral} forms. A form $f$ in $K$ is \emph{abelian} if $a_\ell(f)$ depends only on the congruence class of $\ell$ modulo a power of $2$ (equivalently, $a_\ell(f)$ depends only on $\Frob_\ell$ in an abelian extension of $\QQ$). A form $f$ is \emph{dihedral} if~$a_\ell(f)$ depends only on $\Frob_\ell$ in a dihedral extension of $\QQ$. The graded dihedral forms further fall into two families, \emph{$\QQ(\sqrt{-2})$-dihedral} and \emph{$\QQ(i)$-dihedral}, with the exception of $\Delta$, $\Delta^3$, and $\Delta^5$, which belong to both families. Call a form in $K$ \emph{special} if it belongs to the span of the abelian and the dihedral forms.

Bellaïche shows that the set of special forms is ``thin" in the set of all forms, in the following precise sense.  The dimension of all forms of nilpotence index up to $N$ is \( \frac{1}{2} N(N-1) \), as \( \FF_2\langle m(a,b) \mid a + b < N \rangle \) is a basis.  The space of abelian forms is characterized \cite[Corollaire 17]{Bmem} as being spanned by the forms
\begin{align*}
    \sum_{\substack{a \equiv k \pmod{2} \\ 0 \leq a \leq k}} m_e(a,k-a) \,, \quad k \geq 0 \,, \quad \text{ and } 
    \sum_{\substack{a \not\equiv k \pmod{2} \\ 0 \leq a \leq k}} m_e(a,k-a) \,, \quad k \geq 1, 
\end{align*} 
where \( e = (g, c) \) is a pinning of \( G \).
Hence the  abelian forms with nilpotence index up to~$N$ are a~\( 2N-1\) dimensional space.  Likewise, for any adapted basis, the space of dihedral forms is characterized \mbox{\cite[Corollaire 19]{Bmem}} as being spanned by the forms \( m(a,0) \), \( a \geq 0 \), and \( m(0,b), b \geq 0 \).  The dihedral forms with nilpotence index up to~$N$ also span a \( 2N - 1 \) dimension space.  We note that \( m_e(0,0), m_e(1,0) \) and \( m_e(0,1) \) are already abelian forms, so together the abelian and dihedral forms of nilpotence index up to~$N$ form a \( 4N - 5 \) dimensional space (for \( N > 1 \)).  Since
\[
    4N - 5 \ll \frac{1}{2} N(N-1) \,,
\]
these special forms are indeed ``thin'' in the set of all forms.

\section{Densities of mod-2 modular forms}\label{sec:density}
In this section we take the opportunity to communicate results of Bellaïche on the \emph{density} of a mod-$2$ modular form $f$ --- the proportion of primes $\ell$ for which $a_\ell(f)$ is nonzero.  Some of these results already appear in \cite[\S10]{Bim}; others ---  \cref{item:densityleq1/4} and \cref{bellaiche-density} in particular --- 
have never been published. Most notably \cref{bellaiche-density} gives a surprising and elegant formula for the densities of a   dihedral form in terms of the base-$2$ digits of its nilpotence index. These results will then feed into our eta-power-density computations in \cref{proofs}. All mistakes below are ours.

\subsection{Review of Bellaïche density}\label{belaichedensity}
Following Bellaïche \cite[\S1.2.1, \S10]{Bim}, we define the density $\bel(f)$ of a mod-$2$ modular form $f$ as follows.

\begin{definition}\label{def:beldensity}
Let \( f \) be a modular form modulo~$2$.  Then its density is
\begin{equation}\label{densitydef}\bel(f) \coloneqq  \lim_{n \rightarrow \infty} \frac{\#\{\ell \leq n~ {\rm prime}\mid a_\ell(f) \not\equiv 0 \pmod{2}\} }{\pi(n)}\,.
\end{equation}
 
\end{definition}\smallskip

\begin{remark}
From \cref{dplemma} it follows that $\bel(f) = {\mathcal D}_2(f, 1)$. 
In particular, if $24 \mid r$, then $m_r = 1$, so that $\del(r) = {\mathcal D}_2\big(\eta^r(q), 1\big) = \bel(p_r)$.
\end{remark}

The limit in \eqref{densitydef} exists and is always a rational number, as the association $\ell \to a_\ell(f)$ is \emph{frobenian} in the sense of Serre \cite{S0} (see  \cref{def:frobenian-set}). Indeed, recall that $G_f = G/\ker t_f$ is the quotient of the Galois group of the fixed field of the maximal \mbox{pro-2} extension $E$ of $\QQ$ unramified outside $2$ by the subgroup $\ker t_f$, as defined in \cref{tf}.
Since by the Chebotarev Density Theorem the Frobenius elements equidistribute, the problem of computing $\bel(f)$ reduces to counting the number of certain elements in the Galois group~$G_f.$ More precisely,

\begin{equation}\label{density-calculation-with-chebotarev}
\bel(f) = \frac{\#\{h \in G_f \mid a_1(t(h)f) = 1\}}{\# G_f} = \frac{\#\{h \in G_f \mid f^\ast \circ t_f(h) = 1\}}{\# G_f},
\end{equation}

\smallskip

where $f^\ast: A \to \FF_2$ is the linear form on $A$ corresponding to $f$, sending $T \in A$ to $a_1(Tf)$, established in \cref{subsec:duality}.

A priori it may seem unusual to consider only prime coefficients in $\delta(f)$, but in fact these nearly determine the form. This is a consequence of the following result of Bellaïche.

\begin{theorem}[Bellaïche, {\cite[Theorem 1]{Bim}}]
\label{item:densitygeq0}
  For $f \in K$, we have $\bel(f) > 0$ unless $f = 0,\, \Delta$. 
  \end{theorem}

As a corollary, if $f$ and $g$ are two forms in $K$ with $a_\ell(f) = a_\ell(g)$ for all (or all but a density-zero set of) primes $\ell$, then $f = g$ or $f = g + \Delta$. Therefore prime coefficients all but determine the form, even for noneigenforms. 

If $f$ is a \emph{graded} form, that is $f \in K^i$ for some $i \in (\ZZ/8\ZZ)^\times$, then it is clear that $\bel(f) \leq \frac{1}{4}$. In fact, this inequality is typically strict. 

\begin{theorem}[Bellaïche, unpublished]
\label{item:densityleq1/4} For $f$ in any $K^i$, we have $\delta(f) < \frac{1}{4}$, unless $f = \Delta^3$, $\Delta^5$.
\end{theorem}
\begin{proof}[Proof of \cref{item:densityleq1/4} (Bellaïche)]
First assume that $i = 3$ or $5$. We first consider $\Delta^i$ itself. For $\ell$ prime 
we have $a_\ell(\Delta^i) = 1$ if and only if $\ell \equiv i \cmod{8}$,\footnote{Indeed, one can show that for $i = 3, 5$ we have  $a_n(\Delta^i) = 1$ if and only if $i = \ell^b m^2$, where $\ell$ is a prime congruent to $i$ modulo $8$, $b \equiv 1 \mod{4}$, and $m$ is odd and not divisible by $\ell$.} so that $\bel(f) = \frac{1}{4}.$ 
Now if $f  \in K^i$, then consider $g = f + \Delta^i$. Then $a_\ell(f) = 1$ if and only if $\ell \equiv i \cmod{8}$ and $a_\ell(g) \equiv 0$. Therefore $\delta(f) = \frac{1}{4} - \delta(g)$. \cref{item:densitygeq0} applied to $g$ proves the claim. 

Finally if $i = 1$ or $7$, then we have $\delta(f) = \frac{\# H_f}{\# G_f}$, where $G_f$ as above is a (finite) Galois group and $H_f = \{g \in G_f: a_1(t(g) f) = 1\}.$ We know of course that $H_f \subseteq G_f^i$, the coset of $G_f$ containing all $\Frob_\ell$ with $\ell \equiv i \cmod{8}$. It suffices therefore to show that $H_f \subsetneq G_f^i$. Note that both $1$ and $c$ are never in $H_f$ for any $f \in K$, as $t(1) = t(c) = 0.$ Therefore if $f \in K^1$, then $1 \in G_f^i - H_f$, and if $f \in K^7$, then $c \in G_f^i - H_f$.  Thus in each case $H_f$ is strictly smaller than $G_f^i$, which completes the proof.
\end{proof}

We also record a corollary of \cref{item:densitygeq0,item:densityleq1/4} for any power of $\Delta$. 
\begin{corollary}\label{alldeltapowers} For any $n$ we have $0 \leq \bel(\Delta^n) \leq \frac{1}{4}$, with $\bel(\Delta^n) = 0$ if and only if $n = 1$ or $n$ is even; and $\bel(\Delta^n) = \frac{1}{4}$ if and only if $n = 3,5$. 
\end{corollary}

Recall that a form in $K$ is called special if it is in the span of abelian and dihedral forms. 
Bellaïche has proved formulas for $\bel(f)$ for a basis of special forms, although this work has never been published. In \cref{bellaiche-density} in \cref{densitydihedral} we will give his proof for $\bel(f)$ for dihedral forms (which we will eventually use to prove \cref{thm:main-thm}). From the expression for $\bel(f)$ for dihedral forms, it will already be clear that special forms may have arbitrarily small densities. On the other hand, the a priori expectation is that $\bel(f)$ for all nonspecial forms $f$ is uniformly bounded below.\footnote{In \cite{Bim} Bellaïche has proved the existence of such a uniform bound for forms outside of an exceptional set modulo~$p \geq 3$. But note that his exceptional set is in general bigger than the span of the abelian and dihedral forms.} In fact, Bellaïche has informally conjectured\footnote{in talks (for example, {\scriptsize \url{https://people.brandeis.edu/~jbellaic/mod2columbiabeamer.pdf}}) and private correspondences} the following.
\begin{expectation}[Bellaïche]\label{18}
If $f \in K$ is graded and nonspecial, then $\bel(f) = \frac{1}{8}.$
\end{expectation}
Because of the $(\ZZ/8\ZZ)^\times$-grading, a density of $\frac{1}{8}$ for a graded form is in fact an equistribution result.

\subsection{Densities of dihedral modular forms modulo 2} \label{densitydihedral}

The purpose of this subsection is to present a proof of the Bellaïche's unpublished density computation for a basis of dihedral forms in $K$. Recall that $x = T_3$ and $y = T_5$ in the Hecke algebra $A$, and $\{m(a, b)\}_{a, b}$ is the basis adapted to $x$ and $y$ in the sense \cref{def:adaptedbasis}, and using the shorthand convention of \eqref{eqn:mabdef}.

We set 
\begin{align*}
    d = d(a) &= \text{number of digits of $a$ in base $2$} \,, \\
    v = v(a) &= \text{2-adic valuation of $a$} \,, \\
    z = z(a) &= \text{number of 0s in $a$ base 2} \,, \\
    u = u(a) &= \text{number of 1s in $a$ base 2} \,.
\end{align*}
Note that $d = z + u.$

\begin{theorem}[Bellaïche, unpublished]\label{bellaiche-density} For any $a \geq 0$
	$$\delta\big(m(a,0)\big) = \delta\big(m(0,a)\big) = \frac{1}{2^{u(a)+v(a)+1}}.$$
\end{theorem}

We will prove only the case $m(a,0)$ and the proof for $m(0,a)$ is completely analogous. We fix the pinning~$e=(g,c)$ with $g = \Frob_3$ and $c$ is complex conjugation. Thus, we set $X=x = t(g)$ and~$Y=t(cg)$ and by 
\cref{prop:adaptedbasis} there exists a basis $m_e(a,b)$ adapted to the pinning $e=(g,c).$ From \cref{invariance-of-basis-under-pinning}, we know that \( m_e(a, 0) = m(a, 0) \) for all \( a \geq 0 \).  Now consider $f = m_e(a,0) = m(a,0)$ and $\ann(f) = (X^{a+1},Y)$ and we find for the faithful Hecke algebra quotient
\begin{equation}\label{eqn:af}
A_f = \mathbb{F}_2[x]/(x^{a+1}) \,.
\end{equation}
We compute the Galois group, which is generally given by
$$G_f = \langle g, c \mid c^2 = 1 \rangle / \ker t_f\,.$$

The following is a slight refinement of  \mbox{\cite[Lemma 2]{Bmem}}, to be used to determine $\ker t_f$. 
\begin{lemma}[Bellaïche+$\varepsilon$]\label{sq}Let $B$ be a ring of characteristic 2, and $(t, d): \Gamma \to B$  a pseudorepresentation. Suppose $\gamma \in \Gamma$. If $t(\gamma) = 0$, then $\gamma^2 \in \ker t.$ Conversely, $\gamma^2 \in \ker t$ implies that $t(\gamma)$ annihilates the ideal generated in $B$ by the image $t(\Gamma)$.
\end{lemma}

\begin{proof}
For any $u \in \Gamma$, the trace-determinant identity (see \cref{cond:tr-det}) for the pair $(\gamma, \gamma u)$ gives 
    $$t(\gamma^2 u) - t(u) = t(\gamma) t(\gamma u).$$
Therefore, if $t(\gamma) = 0$, then  $t(\gamma^2 u) = t(u)$ for all $u \in \Gamma$: in other words~$\gamma^2 \in \ker u$. And conversely,~$\gamma^2 \in \ker u$ implies that $t(\gamma) t(\gamma u) = 0$ for every $u \in \Gamma$. In other words $t(\gamma)$ annihilates the image $t(\Gamma),$ as claimed.
\end{proof}

We also introduce modified Chebyshev polynomials (of the first kind), in order to be able to describe the image of $t$ on large cyclic subgroups.  

\begin{lemma}
    For all $n \in \ZZ$ there exists a unique polynomial $S_n \in \Z[x]$ with the property 
    \begin{equation}\label{chubby}S_n(x + x^{-1}) = x^n + x^{-n}.
    \end{equation}

    \begin{proof}
        For $n=0$, we need $S_0(x + x^{-1}) = 2$, so $S_0(x) = 2$ identically.  Since \( S_n(x) \) is a polynomial of some degree, comparison of coefficients in \( S_n(x + x^{-1}) \overset{!}{=} x^n + x^{-n} \) shows that \( S_n(x) \) has degree exactly~\( n \).  Moreover, this leads to a triangular system of equations for \( a_i \), in 
        \[
        S_n(x + x^{-1}) = \sum_{i=0}^{n} a_i (x + x^{-1})^i  \overset{!}{=} x^n + x^{-n} \,.
        \]
        This uniquely determines \( S_n(x) \) for $n \geq 0$. For $n < 0$ we set $S_n(x) \coloneqq S_{-n}(x).$
    \end{proof}
\end{lemma}

\begin{remark} From the definition in \eqref{chubby}, one can easily deduce that  
$S_n(x) = 2 \mathcal T_n ( x )$, where~$\mathcal T_n(x)$ is the Chebyshev polynomial of the first kind, which satisfies
$\mathcal T_n(\cos \theta) = \cos n\theta$. 
\end{remark}

The first few Chebyshev polynomials are readily computed to be
\begin{align*}
    & S_0(x) = 2 \,, \quad S_1(x) = x \,, \quad  S_2(x) = x^2 - 2 \,, \quad S_3(x) = x^3 - 3x \,, \\
    & S_4(x) = x^4 - 4x^2 + 2 \,,\quad  S_5(x) = x^5 - 5x^3 + x \,, \quad  \ldots
\end{align*}

From the definition expression in \eqref{chubby}, one can see that the sequence of polynonomial \( S_n(x) \) satisfies the order-2 recurrence relation \( S_n(x) = x S_{n-1}(x) - S_{n-2}(x)  \) for all $n$.

\begin{prop}\label{chebychev-computation-for-g}
    If $t: \Gamma \rightarrow B$ is a pseudorepresentation of group $\Gamma$ on ring $B$ in the sense of \cref{trace-determinant-identity}, then 
\begin{equation}\label{chebyshev}
t(g^n) = S_n(t(g))
\end{equation}
for all $g \in \Gamma$ and $n \in \Z$.

\begin{proof}
    One checks that both sides agree for $n = 0, 1$, and appeals to the fact that both \( S_n \) and \( t(g^n) \) satisfy an order 2 recurrence relation. For $t(g^n)$ this follows from the trace-determinant identity.
\end{proof}
\end{prop}

As we work in characteristic $2$, we also introduce the mod-$2$ image $\overline S_n(x) \in \FF_2[X]$. In particular, we have
\begin{lemma}\label{chebychev_mod_2}
    For all $n \in \Z$ we have $\overline{S}_{2n} = \overline{S}_n^2.$
    In particular, $\overline S_{2^k} = x^{2^k}$ for all $k\geq 0.$
\end{lemma}

\begin{proof}
We have
 \(  \overline{S}_{2n}(x) = x^{2n} + x^{-2n} = \left(x^n+x^{-n}\right)^2 = \overline{S}_{n}^2.
 \)\qedhere
\end{proof}

We now return to studying the dihedral form $f = m(a, 0)$ for $a > 0$. Recall that $d \coloneqq d(a) \geq 1$ denotes the number of digits in $a$ base $2$, so that $2^d \leq a < 2^{d+ 1}.$

\begin{lemma}\label{kerdi} We have $\ker t_f = \langle cgcg, g^{2^{d + 1}} \rangle$. 
\end{lemma}

(In fact, for the proof of \cref{bellaiche-density} the containment $\ker t_f \supseteq \langle cgcg, g^{2^{d + 1}} \rangle$ suffices.)

\begin{proof}
By \eqref{eqn:af} we have $A_f = \FF_2[ X ]/(X^{a + 1}).$
It follows that $t_f(cg) \equiv Y \pmod{\mm^2}$ is zero in $A_f$, so that by \cref{sq} we have $cgcg \in \ker t_f$. Similarly, $t_f(g^{2^d}) = \overline S_{2^d}(t_f(g)) = (t_f(g))^{2^d} = X^{2^d} = 0$, where the first equality is \cref{chebychev-computation-for-g}, the second equality is  \cref{chebychev_mod_2}), and the last follows because $2^d \leq a + 1$ since $2^d \leq a < 2^{d+1}$ by definition. By \cref{sq} again, we have $g^{2^{d+1}} \in \ker t_f$. Therefore $\ker t_f$ contains $\langle cgcg, g^{2^{d + 1}} \rangle$, so that $G_f$ is a quotient of a group with the presentation  $\langle g, c \mid c^2, (cg)^2, g^{2^{d + 1}}\rangle$: in other words, $G_f$ is a quotient of $D_{2^{d+1}}$, the dihedral group with $2^{d + 2}$ elements. 
In order to show equality we establish that $c$, $cg$, or $g^{2^d}$ are in the kernel of $t_f$.
\begin{itemize}[parsep=1ex, itemindent=2em, labelwidth=2em]
    \item[\fbox{\bf Case $c$\,\llap{\phantom{g}}}] Since $X = t_f(g) \neq t_f(gc) = Y = 0$, we have $c \not\in \ker t_f$. (Here we use $a > 0$, so that $f \neq \Delta$.)
    \item[\fbox{\bf Case $cg$}] Since $0 = t_f(c) \neq t_f(ccg) = t_f(g) = X$, we have $cg \not\in\ker t_f$. (Again, we use $a > 0$ here.)
    \item[\fbox{\bf Case $g^{2^d}$}] By \cref{sq}, if $g^{2^d} \in \ker t_f$, then $t_f(g^{2^{d-1}})$ annihilates the ideal $(X)$ of $A_f.$ (Recall yet again that $a \geq 1$, so that $d \geq 1$.) By \cref{chebychev-computation-for-g} and \cref{chebychev_mod_2} again, we know that $t_f(g^{2^{d-1}}) = X^{2^{d-1}}.$ In other words, the statement $g^{2^d} \in \ker t_f$ implies that $X^{2^{d-1} + 1} = 0$ in $A_f$, or, equivalently, $2^{d-1} \leq a$. But this contradicts the assumption on $d$. \qedhere
\end{itemize}
\end{proof}

We  finally find
\begin{equation}\label{gfdih}G_f = \langle g, c \mid cgc = g^{-1}, c^2 =1, g^{2^{d+1}} = 1 \rangle\,.
\end{equation}
Therefore, $G_f$ is a dihedral group of order $2^{d+2}$. 
We now study the image of $t_f$ on this dihedral group.

\begin{lemma}
    For all $n \geq 0$ we have $t_f(cg^n) = 0$.
\end{lemma}
\begin{proof}
We proceed by induction on $n.$ We first establish the two base cases $n=0$ and $n=1.$
By definition we have $t_f(cg) = Y = 0 \in A_f.$ Furthermore, we have by the trace determinant identity for $t$ that
$$0=t(1)=t(c^2)=t(c)^2+t(c^2)=t(c)^2+t(1)=t(c)^2.$$
Thus, in $A$ we have the equation $t(c)^2 = 0.$ Since $A=\mathbb{F}_2\llb X,Y\rrb$ is an integral domain it follows that $t(c) = 0$ and thus $t_f(c) = 0.$ Now assume that theorem holds for all integers $\leq n$.

From the trace-determinant identity (\cref{cond:tr-det})
$$t_f(cg^{n+1}) = t_f(g)t_f(cg^n) + t_f(g^{-1}cg^n).$$
By assumption the first term vanishes. Applying the same identity again to the last term we obtain
$$t_f(g^{-1}cg^n) = t_f(g^{-1})t_f(cg^n) + t_f(gcg^n).$$
Again the first term vanishes. For the final term we have using the relation $g^{-1} = cgc$ that
$$t_f(gcg^n) = t_f(gcgc^2g^{n-1}) = t_f(cg^{n-1}) = 0.$$
Therefore, $t_f(cg^{n+1})=0$, as claimed.    
\end{proof}

Therefore, we can immediately see that $0 \leq \bel(f) \leq \frac{1}{2}$. If we denote by $[x^a]S_n$ the coefficient of $x^a$ in $S_n$ the upshot is the following key Lemma.

\begin{lemma}
For all $a \geq 0$ and $n\geq 0$ we have that for the forms $f=m(a,0)$
$$f^\ast  t_f(g^n) = [x^a]S_n.$$
\end{lemma}

\begin{proof}
    Firstly, by \cref{chebychev-computation-for-g}, we find
    \[
        t_f(g^n) = S_n(t_f(g)) = S_n(x),
    \]
    since $x=t_f(g)$ by construction.  From the last equality in  \eqref{density-calculation-with-chebotarev} we  have
    \[
        f^\ast \circ t_f(g^n) = a_1\big(S_n(x)m(a,0)\big).
    \]
    We now write out $S_n(x)m(a,0) = c_0m(a,0) + c_1xm(a,0) + \dots + c_nx^nm(a,0)$ with $c_i \in \Z.$ By the choice of our pinning and \cref{invariance-of-basis-under-pinning} it follows that $xm(a,0) = xm_e(a,0) = m_e(a-1,0) = m(a-1,0).$ Thus, we can rewrite
    \begin{equation}\label{chebychev-applied-to-dihedral-form}
        S_n(x)m(a,0) = c_0m(a,0) + c_1m(a-1,0) + \dots + c_nm(a-n,0).
    \end{equation}
    By \cref{def:adaptedbasis:ord} we find that the only term in \eqref{chebychev-applied-to-dihedral-form} which contribute to $a_1(S_n(x)m(a,0))$ is $c_am(0,0)$ (coming from $c_ax^am(a,0)$). Thus, as desired
    \[
        f^\ast \circ t_f(g^n) = [x^a]S_n. \qedhere
    \]
\end{proof}

Thus, the problem is reduced to counting, for fixed $a$, the number of $n$ modulo $2^{d(a) +1}$ such that $x^a$ appears in~$S_n.$ The following theorem, whose proof is given separately in the next section, gives the complete answer. 

\begin{theorem}[Combinatorial Lemma, Bellaïche]\label{combinatorial-lemma}
We have  $[x^a]S_n = 1$ if and only if $a$ is in one of $2^{z(a) - v(a) +1}$ residue classes mod $2^{d+1}.$
\end{theorem}

We include an original proof of \cref{combinatorial-lemma} in \cref{sec:combinatorial-lemma}. 
\cref{bellaiche-density} follows as a corollary of \cref{combinatorial-lemma}.

\begin{proof}[Proof of \cref{bellaiche-density}]
The expression in \eqref{density-calculation-with-chebotarev} specialized to $f = m(a, 0)$ gives 
    \[
    \bel(f) = \frac{2^{z(a)-v(a)+1}}{2^{d(a) + 2}} = \frac{2^{z(a)+1}}{2^{z(a) + u(a) + v(a) + 2}} = \frac{1}{2^{u(a) + v(a) + 1}}.
    \qedhere
    \]
\end{proof}

\section{Proof of the combinatorial lemma}\label{sec:combinatorial-lemma}
In this section we give our proof of \cref{combinatorial-lemma}.

We first need another representation of the Chebyshev polynomials $S_n$ for $n \geq 2.$ We recall that it is easy to show from the defining equation that the $S_n$ satisfy the recurrence $S_n(x) = xS_{n-1}(x) - S_{n-2}(x).$ Furthermore one sees that
\begin{align}\label{chebychev_alternate_representation}
S_n(x) = \sum_{k} (-1)^k \frac{n}{n-k} \binom{n-k}{k} x^{n-2k},
\end{align}
since the polynomial on the right hand side satisfies the same initial conditions and recurrence as the~$S_n.$
As in the previous section, we denote the reduction of $S_n \pmod {2}$ by $\overline{S}_n(x).$ The following two statements are corollaries of \cref{chebychev_mod_2}.

\begin{corollary}\label{corollary1}
    The term $x^a$ appears in $\overline{S}_n(x)$ if and only if the term $x^{2a}$ appears in $\overline{S}_{2n}(x).$
\end{corollary}

\begin{corollary}\label{corollary2}
    If $x^a$ appears as a term in $\overline{S}_n(x),$ then $v(a) = v(n).$ 
\end{corollary}

\begin{proof}
    Let $v = \min\{v(a), v(n)\}.$ Then at least one of the numbers ${a}/{2^v}$ and ${n}/{2^v}$ is odd. By \cref{corollary1}~$x^a$ appears in $\overline{S}_n(x)$ precisely if $x^{{a}/{2^v}}$ appears in $S_{{n}/{2^v}}$.  This implies that ${a}/{2^v}$ and ${n}/{2^v}$ have the same parity since the $S_n$ are either even or odd, depending on the parity of $n.$ However, ${a}/{2^v}$ and~${n}/{2^v}$ having the same parity is equivalent to $v(a) = v(n).$
\end{proof}

Write the base-$2$ expansion of $a$ as $$a = \left[b_{d(a) - 1}(a)  \quad b_{d(a) - 2} \quad \cdots \quad b_1(a) \quad b_0(a)\right]_2.$$
Here $b_i(a) \in \{0, 1\}$ are the base-2 digits of $a$, so that $a = \sum_{i} b_i(a) \cdot 2^i $ and $b_i(a) = 0$ for all $i \geq d(a)$ as well as for $i < 0$. Note that $b_{v(a)}(a)$ is always $1$. Finally, let $Z(a) = \{ 0 \leq i < d(a): b_i(a) = 0\}$ be the places of~$a$ that have a 0 in the base-$2$ expansion. Obviously, $\#Z(a) = z(a).$

We will need the following lemma  relating $v(n)$ with $v\big(\binom{n}{k}\big).$

\begin{lemma}\label{binomial_coefficients_2_adic_properties}
    Let $k$ be odd. We have that $v(n) = v\bigl(\binom{n}{k}\bigr)$ if and only if one of the following two conditions are satisfied
    \begin{enumerate}[topsep = -3pt]
        \item $n$ is odd and $b_i(n) \geq b_i(k)$ for all $i,$\label{part1}
        \item $n$ is even and $b_i(n) \geq b_i(k)$ for all $i > v(n)$ as well as $b_{v(n)}(k) = 0.$\label{part2}
    \end{enumerate}
\end{lemma}

\begin{proof}
First note that 
$$v(n!) = n - (\text{sum of the base-}2 \text{ digits of }n).$$
From this we see immediately that $v\bigl(\binom{n}{k}\bigr)$ is the number of borrowed digits when $k$ is subtracted from $n$ in base $2.$ This is Kummer's theorem; see \cite{Rom} for a convenient reference. 

Since $k$ is odd, we automatically have $ v(n) \leq v\big(\binom{n}{k}\big),$ as the last digit of $k$ is a $1$ and thus for the subtraction of $k$ from $n$ the last $v(n)$ zeros of $n$ are each insufficient and require borrowing from the previous digit. Part \ref{part1} follows immediately. For \ref{part2} we know there are at least $v(n)>0$ borrowed digits and must assure that there are no additional ones. This means first of all that at the $v(n)^{\rm th}$ place there is no borrowed digit, or in other words $b_{v(n)}(k) = 0.$ Here is an example with $v(n) = 4:$
$$
\begin{array}[t]{lrrrrrrrrrrrr}
& n:  & & 1 & \ast &  \ast &  \ast &  \overset{0}{\cancel{1}} & \overset{1}{\cancel{0}} & \overset{1}{\cancel{0}} & \overset{1}{\cancel{0}} & \phantom{}^1 0 \\
- & k:   & & \ast &  \ast &  \ast &  \ast & \ast & \ast & \ast & \ast &  1 \\
\hline
 &    & & \ast &  \ast &  \ast &  \ast & 0 & \ast & \ast & \ast &  1 \mathrlap{\,.}
\end{array}
$$
Also there must be no borrows to the left of the~$v(n)^{\rm th}$ place---that is, every digit of $n$ to the left of the~$v(n)^{\rm th}$ must be no less than the corresponding digit of~$k$. Finally, since these conditions only affect the digits of $n$ to the right of the $\ell(k)^{\rm th}$ place, they depend only on $n$ modulo $2^{d(k)}$.
\end{proof}

We are now ready to give a proof of the Combinatorial Lemma (\cref{combinatorial-lemma}).

\begin{proof}[Proof of \cref{combinatorial-lemma}]
    We will first prove the case where $a$ is odd and then show that the general case reduces to this case.

    Suppose $a$ is odd. We will show that there are exactly $2^{z(a) + 1}$ residue classes mod $2^{d(a) + 1}$ such that $x^a$ appears in $\overline{S}_n(x)$ if and only $n$ is in one of those $2^{z(a) + 1}$ residue classes.

    By reducing \eqref{chebychev_alternate_representation} mod $2$ we find that
    $$\overline{S}_n(x) = \sum_{k} \frac{n}{n-k} \binom{n-k}{k} x^n = \sum_{a \equiv n \pmod{2}} \frac{2n}{n+a} \binom{\frac{n+a}{2}}{a} x^a.$$
    Since the sum runs over all $n$ with the same parity as $a$ we are thus interested in all $n$ such that $\frac{2n}{n+a}\binom{\frac{n+a}{2}}{a}$ is odd. We have thus reduced the problem to counting all numbers $m = \frac{n+a}{2}$ with the property that $v(m) = v\left(\binom{m}{a}\right).$ We will now use \cref{binomial_coefficients_2_adic_properties}, which tells us that either $v(m) = 0$ or $b_{v(m)}(a) = 0.$ Hence, we conclude $v(m) \in Z(a) \cup \{0, d(a)\}.$ Given a fixed choice of $v(m)$ we have by definiton $b_{v(m)}(m) = 1$ and all the digits in the base $2$ expansion of $m$ right of the $v(m)^{\rm th}$ are $0.$ By \cref{binomial_coefficients_2_adic_properties} we find that for all~$i > v(m),$ if $b_i(a) = 1$ it implies that $b_i(m) = 1$ and if $b_i(a) = 0,$ then $b_i(m)$ can either be $0$ or~$1.$

    For each index $i \in Z(a) \cup \{0, d(a)\} = \{ d(a) > i_{z(a)} > \dots > i_2 > i_1 > 0\},$ one can unqiuely write down one $m$ with $v(m) = i$ for every subset of zeros of $a$ appearing to the left of $i^{\rm th}$ place. There is one $m$ with $v(m) = d(a),$ one~$m$ with $v(m) = i_{z(a)},$ $2$ allowed $m$ with $v(m) = i_{z(a) - 1},$ and so on through $2^{z(a)}$ allowed~$m$ with $v(m) = 0.$ In total we get
    $$1 + 1 + 2 + \cdots + 2^{z(a)-1} + 2^{z(a)} = 2^{z(a) + 1}$$
    allowed $m$ modulo $2^{d(a)}.$ If we now translate from $m= \frac{n+a}{2}$ to $n = 2m - a$, we find exactly $2^{z(a) + 1}$ residue classes mod $2^{d(a)}$ where $x^a$ appears in $\overline{S}_n(x).$ So this case $a$ odd is proven.

    We now show that the general case can be reduced to $a$ being odd: write $a = 2^{v(a)}a'$, where $a'$ is odd. We know by \cref{corollary1,corollary2} that $x^a$ appears in $\overline{S}_n(x)$ if and only if $n = 2^{v(a)}n'$ with $x^{a'}$ appearing in $\overline{S}_{n'}(x).$ This happens if $n'$ lies in one of the $2^{z(a') + 1}$ admissible residue classes mod $2^{d(a') + 1}.$ Since $z(a') = z(a) - v(a)$ and $d(a') = d(a) - v(a),$ this gives as claimed that $n$ lies in one of $2^{z(a)-v(a) + 1}$ admissible residue classes mod $2^{d(a) + 1}.$
\end{proof}

\section{Modular forms mod 2 of level 9}\label{sec:lvl9}
We extend some of the results of Nicolas, Serre, and Bellaïche on mod-2 modular forms of level~1 to level~$\Gamma_0(9)$. The general outline of the story is similar, so we proceed efficiently, only indicating those features of the proofs that differ substantially from the level-1 case. 
 
\subsection{The space of forms, their Hecke algebra, and the Galois pseudodorepresentation}
Here we generalize \cite{NS1,NS2,Brep}, with a bent towards understanding powers of $C$ modulo $2$. Recall that~$C$ is the mod-$2$ $q$-expansion of the unique normalized cuspform of weight $4$ and level $9$.

\begin{enumerate}[itemsep = 2pt, listparindent = 0pt, parsep = 5pt, leftmargin = 2em]
\item {\bf Algebra of forms:} 
As before, we let $M_k(9)\coloneqq M_k(9, \FF_2)$ be the image in $\FF_2\llbracket q \rrbracket $ of mod-$2$ $q$-expansions of forms of weight $k$ and level $9$ with integral Fourier coefficients at the cusp at infinity, and set $M(9) = \sum_k M_k(9)$.
Since $$E_{2, 3}(q) \coloneqq \frac{1}{2}\big(E_2(q) - 3 E_2(q^3)\big) = 1 + 12 q + 36 q^{2} + 12 q^{3} + 84 q^{4} + 72 q^{5} + O(q^{6})$$
is a (holomorphic) modular form of level 3 (and hence of level 9) with mod-$2$
 $q$-expansion $1$, we actually have $M_k(9) \subseteq M_{k + 2}(9)$, so that $M(9) = \bigcup_k M_k(9)$, a filtered $\FF_2$-algebra under the weight filtration $w$. Here $E_2(q) = 1 - 24 q - 72 q^{2} - 96 q^{3} - 168 q^{4} - 144 q^{5} + O(q^{6})$ is the usual (nonholomorphic) Eisenstein series of weigh $2$. 

 \item {\bf Structure of $M(9)$:}
We claim that $M(9) = \FF_2[F],$
where $$F = \sum_{3 \nmid n} q^{n^2} = q +  q^{4} +  q^{16} +  q^{25} +  q^{49} +  q^{64} +  O(q^{100}) \in M_2(9).$$
For example, $C = F + F^4$ and $\Delta = F + F^4 + F^9 + F^{12}.$

Indeed, let $\chi$ be the quadratic character of conductor $3$ viewed with modulus $9$. Then a basis for~$M_1(
\Gamma_1(9), \chi)$ is given  by $e(q) \coloneqq E_{1, 1, \chi}(q)$ and $e(q^3)$ (see \cite[{Theore~5.8\footnote{The full text is freely available.  For Theorem 5.8 see \scriptsize\url{https://wstein.org/books/modform/modform/eisenstein.html}.}}]{St} for both the statement and the notation). It is easy to see that $6 e(q)$ is a lift of the mod-2 Hasse invariant to weight $1$ with character $\chi$, and that $e(q) - e(q^3)$ is a lift of $F$ to weight~$1$ with character $\chi$. In particular, if we let~$w_1$ be the filtration on $M(9)$ when we view each form as a mod-2 form of level~$\Gamma_1(9)$, then~$w_1(F) = 1$. (See, for example, \cite[\S 2.4]{deomed} for context and details on $w_1$.) It follows that a lift of~$F^n$ appears for~$\Gamma_1(9)$ with character $\chi^n$ in weight $n$. Since $\chi$ is quadratic, we have $w(F^n) = n$ for $n$ even; since further Hasse lifts with the same quadratic character, we also have $w(F^n) = n + 1$ for $n$ odd. Therefore~$M_k(9)$ contains $1, F, \ldots, F^{k}$.  
Standard dimension formulas give $\dim M_k(9) = k + 1$ for~$k \geq 0$ even, 
completing the proof.

\item {\bf The intersection of the kernels of $U_2$ and $U_3$:}
Let $K(9) \subset M(9)$ be the intersection of the kernels of $U_2$ and $U_3$. This is the space that will be in duality with the Hecke algebra. The Hecke recurrence for the sequence $\{U_2(F^n)\}_n = \{1, F^2, F, F^3, F^2, F^4, F^3, F^5 \ldots\}$ has order $2$ and hence characteristic polynomial $X^2 + F$, whence it is easy to determine that a $\ker U_2$ is spanned by $\{F + F^4,\ F^3 + F^6,\ F^5 + F^8,\ F^7 + F^{10},\ \ldots, \} = \{F^n + F^{n + 3}\}_{n\ {\rm odd}}.$
(See \cite[Théorème~3.1]{NS1} or \cite[Chapter 6]{Mthesis} for more on the Hecke recurrence.)

The same method for $U_3$ gives us 
$\{U_3(f^n)\}_n = \{1,\ 0,\ 0,\ F^3 + F^2 + F,\ 0,\ 0,\ F^6 + F^4 + F^2,\ 
\ldots\}$, so that the order-3 characteristic polynomial is $X^3 + F^3 + F^2 + F$, and $\ker U_3$ has basis $\{F^n\}_{3 \nmid n}$.

It follows that a basis for $K(9)$ is given by $\{F^n + F^{n + 3}\}_{\{n: \gcd(n, 6) = 1\}}$.

\item {\bf The grading on $K(9)$:} By \cite[Theorem~3.4]{deomed} $K(9)$ has a grading by $(\ZZ/24\ZZ)^\times$: that is, 
$$K(9) = K(9)^1 \oplus K(9)^5 \oplus K(9)^7 \oplus K(9)^{11} \oplus K(9)^{13} \oplus K(9)^{17} \oplus K(9)^{19} \oplus K(9)^{23},$$
with $K(9)^i \coloneqq  \{f \in K(9): a_n(f) \neq 0 \implies n \equiv i \pmod{24}\}.$ Here of course $i$ is viewed modulo $24.$

Moreover, for every prime $\ell \geq 5$, we have 
\begin{equation}\label{hecke9}T_\ell K(9)^i \subseteq K(9)^{\ell i}.
\end{equation}

\item {\bf The Hecke algebra:} For $k \geq 2$ even let $A_k(9)$ be the algbera generated inside $\End_{\FF_2}\big(M_k(9) \cap K\big)$ by the action of the Hecke operators $T_\ell$ with $\ell \geq 5$. As in the level-$1$ case, let $A(9) \coloneqq \varprojlim_k A_k(9)$, where $A_{k+2}(9) \twoheadrightarrow A_{k}(9)$ by restriction.  Then $A(9)$ is a complete noetherian $\FF_2$-algebra.

\item {\bf Duality:} \label{duality9} As in level 1, we have a continuously perfect pairing $A(9) \times K(9) \to \FF_2$ given by $\langle T, f \rangle \mapsto a_1(Tf)$ that in particular induces a natural $A(9)$-equivariant isomorphism $K(9) = A(9)^{\vee, {\rm cont}}$, through which we may identify a form $f$ in $K(9)$ with closed hyperplane $H_f$ of $A(9)$ corresponding to the kernel of pairing with $f$. 

\item {\bf The unique maximal ideal of $A(9)$:} \label{local9} One can show that $A(9)$ is a \emph{local} ring whose maximal ideal~$\mm(9)$ contains all the $T_\ell$ with $\ell \geq 5$. These follow from Serre Reciprocity (Serre's conjecture) and the fact that the unique Hecke eigensystem in $M_k(9)$ for $k = 2,4$ is $\ell \mapsto 0$.

By construction, $\mm(9)$ is certainly the closed ideal of $A(9)$ generated by the $T_\ell$. But by the argument of \cite[Lemma 10.2.4]{Bim}, $\mm(9)$ is also the closed \emph{$\FF_2$-subvector-space} of $A(9)$ generated by all the~$T_\ell$ with~$\ell \geq 5$.

\item {\bf Powers of $C$ in $K(9)$:}
From the expression for $C$ in \eqref{Ceq} and the fact that $C = F + F^4$, it is clear that $C \in K(9)^1$. It follows that $C^i \in K(9)^i$ for every $i$ relatively prime to $24$, so that $K(9)^i$ contains the span of $\{C^n: n \equiv i \pmod{24}\big\}.$\footnote{In fact, Monsky has shown that the space $\mathcal C := \FF_2\big\langle C^i: \gcd(i, 24) = 1 \big\rangle$ is a proper $A(9)$-stable submodule of~$K(9)$ and has determined the structure of the faithful quotient of $A(9)$ acting on $\mathcal C$ \cite[Theorem~2.17 and \S 5]{Monsky}.}

\item {\bf The Galois pseudorepresentation:} As in level $1$, we have a continuous Galois pseudorepresentation 
$t_9: \Gal(\overline\QQ/\QQ) \to A(9)$
that is odd (that is, $t(c) = 0$ for any complex conjugation $c$) and, for any prime $\ell \geq 5$, unramified with $t(\Frob_\ell) = T_\ell$. (In fact, it follows from \eqref{hecke9} that $A(9)$ and~$t$ behave compatibly with the $(\ZZ/24\ZZ)^\times$-grading in the sense of \cite[Theorem~3.4]{deomed}.)

Analogously to \cite[Première étape of the proof of Théorème~1]{Brep}, $t$ factors through the maximal pro-$2$ extension $E_9$ of $\QQ$ unramfied outside $2$ and $3$. 
\end{enumerate}
\subsection{Fields of determination for mod-2 modular forms of level 9}
In this subsection, we translate a very small part of the analysis of \cite{Bmem} to level $9$. More work remains to be done: determining the precise structure of the Hecke algebra $A(9)$ and the Galois group $\Gal(E_9/\QQ)$, as well a thorough understanding of abelian and dihedral forms. For our purposes we content ourselves with the following.  

Fix $f \in K(9)$. By the same argument as in level $1$, the association $\ell \mapsto a_\ell(f)$ is frobenian (see \cref{def:frobenian-set}), and the minimal field of determination $E_f$ is a subfield of $E_9$, so in particular the corresponding Galois group $G_f= \Gal(E_f/\QQ)$ is a $2$-group. 

As in level $1$, we call a form $f \in K(9)$ \emph{abelian} if its minimal field of determination $E_f$ is abelian over $\QQ$. (We may similarly define \emph{dihedral} forms of level 9, but we will not pursue this topic further here.)

\begin{prop}\label{ab9}
Define the following six forms $\aa_i \in K(9)^i$ for $i \in (\ZZ/24\ZZ)^\times - \{\pm 1\}$: 
\begin{align*}
    \aa_5 &= C^5, & \aa_7 &= C^7, & \aa_{11} &= F^{20} + F^{17} + F^{14} + F^{11}, \\
    \aa_{13} &= C^{13},& 
    \aa_{17} &= F^{20} + F^{17}, & 
    \aa_{19} &=F^{28} + F^{25} + F^{22} + F^{19}.
\end{align*}
Then $\aa_i$ is abelian with field of determination $\QQ(\mu_{24})$.
Furthermore, for primes $\ell \geq 5$ we have $$a_\ell(\aa_i) = 1 \iff \ell \equiv i \pmod{24}.$$ 
\end{prop}

\begin{proof}
We proceed as follows: for each $i$, using two of the three formulations 
$$C = \sum_{\gcd(n, 6) = 1} q^{n^2}, \qquad
C^3 = \Delta(q^3) = \sum_{\text{$n$ odd}} q^{3n^2},\qquad \text{ and } \qquad  F = \sum_{3 \nmid n} q^{n^2},$$
we express \( \alpha_i \) as 
\begin{equation} \label{alphai} \alpha_i = \sum_{\substack{m, n > 0 \\ \text{+ conditions}}} q^{Q(m,n)}
\end{equation} for some quadratic form $Q(x,y) = ax^2 + by^2$ of fundamental discriminant $-D$. By considering factorization in $K\coloneqq \QQ(\sqrt{-D})$ we conclude that a prime $\ell \equiv i \pmod{24}$ is represented by $Q(x,y)$ uniquely. Separately, by considering squares modulo $24$, we show that $Q(m, n) \equiv i \pmod{24}$ if and only if $m$ and $n$ satisfy the precise conditions in the expression for $\alpha_i$ from \eqref{alphai}. 
The case $i=11$, for which the question of representability of $\ell$ by $Q(x,y)$ is the most delicate, as the ring of integers of $K$ is not a PID, is given in full below.  The other cases are straightforward. 
\renewcommand{\arraystretch}{1.2}
\begin{center}
\begin{tabular}{c|c|c|c|c|c}
$i$ & $\alpha_i$ & $Q(m, n)$ & condition on $m$ & condition on $n$ & field $K$ \\
\hline\hline 
5 & $C^4 \cdot C$ & $4m^2 + n^2$ & $\gcd(m, 6) = 1$ & $\gcd(n, 6) = 1$ & $\QQ(i)$\\
7 & $C^4 \cdot C^3$ & $4m^2 + 3n^2$ & $\gcd(m, 6) = 1$ & $n$ odd & $\QQ(\sqrt{-3})$\\
11 & $C^3 \cdot F^8$ & $3m^2 + 8n^2$ & $m$ odd &  $3 \nmid n$ & $\QQ(\sqrt{-6})$\\
13 & $C \cdot (C^3)^4$ & $m^2 + 12n^2$ & $\gcd(m, 6) = 1$ & $n$ odd & $\QQ(\sqrt{-3})$\\
17 & $F^{16} \cdot C$ & $16m^2 + n^2$ & $3 \nmid m$ & $\gcd(n, 6) = 1$  & $\QQ(i)$\\
19 & $C^3 \cdot F^{16}$ & $3m^2 + 16n^2$ & $m$ odd & $3 \nmid n$ & $\QQ(\sqrt{-3})$
\end{tabular}\end{center}
For example, for $i = 11$, we have 
\begin{align*}
    \alpha_{i} = F^{20} + F^{17} + F^{14} + F^{11} &= (F^4 + F)(F^{16} + F^{10}) = (F^4 + F)(F^8 + F^5)^2 \\ &= (F^4 + F)(F^4 + F)^2 F^8 = C^3 \cdot F^8  = \sum_{\substack{\text{$m$ odd} \\ 3 \nmid n}}
q^{3m^2 + 8n^2}.
\end{align*}
By considering squares modulo $24$, is easy to confirm that $3m^2 + 8n^2 \equiv 11 \pmod{24}$ if and only if $m$ is odd and $n$ is prime to $3$. To show that any prime $\ell \equiv 11 \pmod{24}$ is represented by $Q(x, y) = 3x^2 + 2y^2$ uniquely, we can argue as follows: since $-6 $ is a square modulo $\ell$, the ideal~\((\ell)\) splits in $K = \QQ(\sqrt{-6})$ into a product of two primes, each of norm $\ell$. Let $\lambda$ be one of the primes lying over $\ell$. Since \( \ell = x^2 + 6y^2 \) has no solutions (consider modulo $6$, for example), $\lambda$ is nonprincipal: in fact, $\lambda= (\ell, a +  \sqrt{-6})$, where~$a$ is a square root of $-6$ modulo $\ell$. Similarly, we can factor the ramified ideal $(2)$ into the square of the nonprincipal ideal $\mathfrak p = (2,  \sqrt{-6})$. Since $K$ has class number $2$, the ideal $\mathfrak p \lambda = (b + m \sqrt{-6})$ is principal, with $b^2 + 6m^2 = 2\ell$, for unique positive $b$, $m$ (since $K$ has no nontrivial units). Moreover, $b$ must be divisible by $4$ (consider modulo $8$). Therefore $n\coloneqq b/4$ and $m$ are the unique positive integral solutions to $Q(m, n) = \ell$ that we seek. 
See also \cite[\S2C, esp.~Cor.~2.27; \S3 esp.~Cor.~3.26]{Cox}.
\end{proof}
 
\subsection{Densities of mod-2 modular forms of level 9}\label{densityitem}
In this subsection, we extend \cref{belaichedensity} to level~9, which amounts to applying results of \cite[Chapter 10]{Bim} and extending some unpublished work of Bellaïche. As with forms of level $1$, for a form $f \in K(9)$, we consider $\bel(f)$. 

\begin{theorem}\leavevmode
\begin{enumerate}[itemsep = 3pt, topsep=-5pt, ref = {\thetheorem(\roman*)}]
\item {\bf Dyadicity:} For any $f \in K(9)$, the density $0 \leq \bel(f) \leq 1$ is a dyadic rational number.
\item\label[theorem]{thm:lvl9:nonzero}  \label{cond:non-zero}{\bf Density is nonzero (with one exception):} For $f \in K(9),$ we have $\bel(f) = 0$ if and only if~$f = C$. In other words, for any $f \neq C$ in $K(9)$ we have $\bel(f) > 0.$ 

\item{\bf Graded density is nonmaximal (with six exceptions):} 
For any \emph{graded} $f$ in $K(9)$ (that is, $f \in K(9)^i$ for some $i \in (\ZZ/24\ZZ)^\times$), we have $\bel(f) \leq \frac{1}{8}.$ If $f \neq \aa_5$, $\aa_7$, $\aa_{11}$, $\aa_{13}$, $\aa_{17}$, $\aa_{19}$ from  \cref{ab9}, then $\bel(f) < \frac{1}{8}$. 
\end{enumerate}
\end{theorem}

\begin{proof}\leavevmode
\begin{enumerate}[itemsep = 3pt]
\item The association $\ell \mapsto a_\ell(f)$ is frobenian with field of determination~$E_f$ a subfield of $E_9$, hence $2$-power-order over $\QQ$, and the expression in \eqref{density-calculation-with-chebotarev} makes it clear that the denominator of~$\bel(f)$ divides $[E_f: \QQ].$ 

\item See Bellaïche \cite[Theorem~I]{Bim}. 

\item The argument is analogous to the proof of \cref{item:densityleq1/4}. First of all, it is a priori clear that $\bel(f) \leq \frac{1}{8}$ for any graded~$f$ in $K(9)$. Next, the description in \cref{ab9} makes it clear that $\bel(\aa_i) = \frac{1}{8}$ for each $i \equiv \pm 5, \pm 7, \pm 11$ modulo $24$. On the other hand, under the assumption that $f \neq \aa_i$ we have $g = f + \aa_i \neq 0$ and $a_\ell(g) = 0 \iff a_\ell(f) = 1$ for primes $\ell \equiv i \pmod{24}$ so that $\bel(g) = \frac{1}{8} - \bel(f)$. Since $\bel(g) > 0$ we have $\bel(f) < \frac{1}{8}.$ And if $i \equiv 1$ (respectively, $23$) modulo $24$, then any $g$ in the conjugacy class of $1$ (respectively, a complex conjugation $c$) in the $i$-graded piece of~$G_f$ has $t(g) = 0$, so that a nonzero proportion of elements of $G_f$ map to the kernel of $f^\ast$, so that 
$\bel(f) < \frac{1}{8}.$\qedhere 
 \end{enumerate}
\end{proof}

We record a corollary analogous to \cref{alldeltapowers}.
\begin{corollary}\label{allcpowers}
For all $n \geq 1$ we have $0 \leq \bel(C^n)\leq \frac{1}{8},$ with $\bel(C^n) = 0$ if and only if $n = 1$ or $\gcd(n, 6) > 1$; and $\bel(C^n) = \frac{1}{8}$ if and only if $n = 5, 7, 13.$
\end{corollary}

\begin{proof}
It remains to prove only that $\bel(C^n) = 0$ whenever $3 \mid n$. This follows from the grading: since~$C$ is supported on the arithmetic sequence $1$ mod $24$, it follows that $C^n$ is supported on the arithmetic sequence $n$ mod $24$. This sequence contains at most finitely many primes whenever $\gcd(n, 6) > 1$. 
\end{proof}

\section{Proofs of the main theorems}\label{finalsection}\label{proofs}

In this section we prove \cref{thm:densityyes,,thm:zero-density,,thm:1/2,,thm:main-thm} stated in the introduction.

\subsection{Key lemmas and the proof of \texorpdfstring{\cref{thm:densityyes}}{Theorem \ref{thm:densityyes}}}

In this subsection we give a final reinterpretation of our notion of eta power density (\cref{def:subdensity}) in terms of Bellaïche density $\bel(f)$. We deduce that our density is always well defined. 

\begin{prop}\label{tellf}
Fix $r$; let $f := P_r$, $m := m_r$, and $b := b_r.$
For $c$ with $\gcd(c, m) = 1$, let $1\leq u_c< m$ be the least integer congruent to $b c^{-1}$ modulo $m$. Then 
$$\del(r) = \sum_{\gcd(c, m) = 1} \bel(T_{u_c} f).$$
\end{prop}

\begin{proof}
This follows directly from \cref{dplemma} and the fact that for $K$ is a graded $A$-module and $K(9)$ is a graded $A(9)$-module, so that the conditions on $\ell$ are unnecessary. 
\end{proof}
    
\begin{corollary}\label{explicit}
    
We have the following expressions for $\del(r)$ in terms of Bellaïche density. 

\renewcommand{\arraystretch}{1.3}
\begin{center} \begin{tabular}{l|C|C|l|l}
description of $r$ & \gcd(r, 24)  & m_r & 
$f = P_r$ & $\del(r)$\\ 
\hline\hline
$r = 3s$ with $s$ odd & 3 & 8 & 
$\Delta^s$ & 
$\bel(f) + \bel(T_3 f) + \bel(T_5 f) + \bel(T_7 f)$
\\\hline
$r = 6s$ with $s$ odd & 6 & 4 & 
$\Delta^s$ & 
$\bel(f) + \bel(T_3 f)$
\\\hline
$r = 12 s$ with $s$ odd & 12 & 2 & 
$\Delta^s$ & $\bel(f)$
\\\hline
$r = 24 s$ & 24 & 1 & 
$\Delta^s$ & $\bel(f)$
\\\hline\hline
\multirow{2}{*}{$r$ prime to $6$} & \multirow{2}{*}{$1$} & \multirow{2}{*}{$24$} & 
\multirow{2}{*}{$C^r$} & 
$\bel(f) + \bel(T_5 f) + \bel(T_7 f) + \bel(T_{11} f) + $\\
& & & & $\qquad + \bel(T_{13}f) + \bel(T_{17} f) + \bel(T_{19} f) + \bel(T_{23} f)$
\\\hline
$r = 2s$ with $s$ prime to $6$ & 2 & 12 & 
$C^s$ & 
$\bel(f) + \bel(T_5 f) + \bel(T_7 f) + \bel(T_{11} f)$
\\\hline
$r = 4s$ with $s$ prime to $6$ & 4 & 6 & 
$C^s$ & 
$\bel(f) + \bel(T_5 f)$
\\\hline
$r = 8s$ with $3 \nmid s$ & 8 & 3 & 
$C^s$ & 
$\bel(f) + \bel(U_2 f)$ 
\end{tabular}
\end{center}
\end{corollary}

\begin{proof}
The proof is a combination of \cref{tellf} and \cref{linktodelta}.
\end{proof}

\begin{proof}[Proof of \cref{thm:densityyes}]
Follows from \cref{explicit}. \end{proof}

\begin{remark}\label{remark7.4}
More generally one can use the expression in \cref{dplemma} to show that, for any integral characteristic-zero modular form $f$ supported on an arithmetic progression $b$ modulo $m$, and any prime~$p$, the limit defining the density $\mathcal D_p(f, m)$ in  \cref{1} exists and is rational. Indeed, the set $S_c := \{\ell \equiv c \pmod{m}: a_\ell(T_{u_c} f)  \not\equiv 0 \pmod{p}\}$ from \cref{dplemma} is the intersection of two frobenian sets, the $\QQ(\zeta_m)$-frobenian set $\{\ell\equiv c \mod{m}\}$ and the set $\{\ell: a_\ell(T_{u_c} f)\not\equiv 0 \pmod{p}\}$, frobenian with respect to any field of determination for $T_{u_c} f$. Since the intersection of two frobenian sets is frobenian with respect to the compositum of the two relevant fields, we conclude.
\end{remark}

\subsection{Proof of \texorpdfstring{\cref{thm:zero-density}}{Theorem \ref{thm:zero-density}}}
In this subsection we establish a precise condition when our notion of density of eta powers (\cref{def:subdensity}) vanishes.

We first establish that $\del(r) = 0$ for $r = 1,2,3,4,6,8,12,16, 24$ (so $r$ properly dividing $32$ or $48$) and for $r$ of the form $16s$ for $s \equiv 0, \pm 2, 3$ modulo~$6$ (in other words, $r$ a multiple of $32$ or $48$). 
For $r$ dividing 24 we give an direct proof to show that in specific cases the density can be computed with  elementary analytic methods.
Indeed, for $r=1,2,3,4,6,8,12,24$ we have $b_r=\frac{r}{\gcd(24,r)}=1$ and thus by \cref{linktodelta} we have that
\[
     P_3 \equiv P_6 \equiv P_{12} \equiv P_{24} \equiv \Delta \pmod{2}.
\]
\[
    P_1 \equiv P_2 \equiv P_4 \equiv P_8 \equiv C \pmod{2}.
\]
Recall now that
\[
    \Delta(q) \equiv \sum_{\text{$n$ odd}} q^{n^2}, \quad
    C(q) \equiv \sum_{\substack{\text{$n$ odd} \\ 3 \nmid n}} q^{n^2}.
\]
Now let $x$ be any positive real number. It is easy to see that the maps $\ell \mapsto \delta_{{\ell,r}}$ are injective in $\ell$ for all~$r$ and bounded in size by $O(x)$ uniform in $r$. Both series are only supported on at most the odd squares. By the Prime Number Theorem there are asymptotically $\frac{x}{\log x}$ primes $\ell$ up to $x$, and $O(\sqrt{x})$ square numbers in the same range.  Hence, we have that
\[
    \frac{\#\{\text{$\ell \leq x$ prime} \mid p_r(\delta_{\ell,r}) \equiv 1 \pmod 2\}}{\pi(x)} = O\left(\frac{\sqrt{x}}{\frac{x}{\log x}}\right) = O\left(\frac{\log x}{\sqrt{x}}\right) \overset{x \rightarrow \infty}{\longrightarrow} 0.
\]

An analogous argument applies to $r=16$, as in this case have $P_{16}(q) \equiv C^2(q) \equiv C(q^2)$, which is supported on at most twice the odd squares. 

In the case $r=32n$ one finds that $P_{32n}(q) \equiv C^{4n}(q) \equiv C^{n}(q^4)$. Therefore, the $q$-series of $P_{32n}(q)$ is a $q$-series in $q^4$. However, one can easily see that $\delta_{32n,\ell}$ can never be a multiple of $4$.

If we have $r=48n$ again by \cref{linktodelta} we do have $P_r \equiv \Delta^{2n}$. Moreover, it is easy to see that we have in this case that $\mu_{\ell, r} = 1$ and $\delta_{\ell,r}$ is odd while the $q$-series $\Delta^{2n}$ modulo $2$ is supported only on even powers, this proves the claim.

We now show the converse. Let $r \geq 1$ be such that $\del(r)=0$. We claim that \[r \in \{1,2,3,4,6,8,12,16,24\} \cup \{32n,48n \mid n \geq 1\}.\] By  \cref{linktodelta} we have that $P_r \equiv \Delta^{b_r}$ or $P_r \equiv C^{b_r}$ modulo $2$.  

\fbox{\( 3 \mid r \)} In this case we have $P_r \equiv \Delta^{b_r}$. We know by \cref{alldeltapowers} that $\bel(\Delta^{b_r}) = 0$ if and only if $b_r=1$ or $b_r$ is even. The case $b_r=\frac{r}{\gcd(24,r)}=1$ occurs precisely if $r\in \{3,6,12,24\}$. On the other hand we see that $b_r$ can only be even if $16 \mid r$ and since by assumption we also have $3 \mid r,$ we find that in this case one must have $48 \mid r$.

\fbox{\( 3 \nmid r \)} Note that $3 \nmid b_r$ as well. By \cref{allcpowers} we know that for such $b_r$ we have $\bel(C^{b_r}) = 0$ only if~$b_r = 1$ or $b_r$ is even. As before we can have $b_r = 1$ only if $r \in \{1,2,4,8\}$. For $b_r$ to be even we must have~$16 \mid r$. Suppose $r = 16n$ with $n > 1$ odd (and still prime to $3$), so that $b_r = 2n$. Then, by \cref{explicit}
\[
\del(r) = \bel(C^{2n}) + \bel(U_2 \,C^{2n}) = \bel(C^n) > 0, 
\]
because $U_2$ is a left inverse for the map $f \mapsto f^2$ modulo $2$ and 
\cref{thm:lvl9:nonzero} guarantees that $\bel(C^n) >0$. Thus if $\del(r)=0$ with $b_r$ even we must have $32 \mid r$. This finishes the proof.
\hfill\qedsymbol

\subsection{Proof of \texorpdfstring{\cref{thm:1/2}}{Theorem \ref{thm:1/2}}}
To show each of the claims in this theorem, we apply the explicit formulas for \( \del(r) \) obtained in \cref{explicit}, to write \( \del(r) \) in terms of Bella\"iche densities.  We recall that in level~$1$ we have  \( \bel(\Delta^n) \leq \frac{1}{4} \) by virtue of the \( (\Z/8\Z)^\times \)-grading, and in level $9$ we have \( \bel(C^n) \leq \frac{1}{8} \) by virtue of the \( (\Z/24\Z)^\times \)-grading.
This means we can immediately update the table from \cref{explicit} with an  additional column giving a bound on \( \del(r) \), as follows.
\renewcommand{\arraystretch}{1.3}
\begin{center} \begin{tabular}{l|l|l|l}
description of $r$ & 
$f = P_r$ & $\del(r)$ & $\del(r)$ bound\\ 
\hline\hline
$r = 3s$ with $s$ odd & 
$\Delta^s$ & 
$\bel(f) + \bel(T_3 f) + \bel(T_5 f) + \bel(T_7 f)$ 
& $\leq1$
\\\hline
$r = 6s$ with $s$ odd & 
$\Delta^s$ & 
$\bel(f) + \bel(T_3 f)$ &
$\leq\frac{1}{2}$
\\\hline
$r = 12 s$ with $s$ odd & 
$\Delta^s$ & $\bel(f)$&
$\leq\frac{1}{4}$
\\\hline
$r = 24 s$ & 
$\Delta^s$ & $\bel(f)$&
$\leq\frac{1}{4}$
\\\hline\hline
\multirow{2}{*}{$r$ prime to $6$} & 
\multirow{2}{*}{$C^r$} & 
$\bel(f) + \bel(T_5 f) + \bel(T_7 f) + \bel(T_{11} f) + $
& \multirow{2}{*}{$\leq1$} \\
& & $ \qquad + \bel(T_{13}f) + \bel(T_{17} f) + \bel(T_{19} f) + \bel(T_{23} f)$ & 
\\\hline
$r = 2s$ with $s$ prime to $6$ & 
$C^s$ & 
$\bel(f) + \bel(T_5 f) + \bel(T_7 f) + \bel(T_{11} f)$&
$\leq\frac{1}{2}$
\\\hline
$r = 4s$ with $s$ prime to $6$ & 
$C^s$ & 
$\bel(f) + \bel(T_5 f)$&
$\leq\frac{1}{4}$
\\\hline
$r = 8s$ with $3 \nmid s$ & 
$C^s$ & 
$\bel(f) + \bel(U_2 f)$ &
$\leq\frac{1}{4}$
\end{tabular}
\end{center}

We now go through the points of \cref{thm:1/2} case by case. 

\fbox{\( \del(n) < 1 \)}  It is clear from the table that \( \del(n) \leq 1 \).  For \( \del(n) \overset{?}{=} 1 \) to occur, we must have either: \( n = 3s \) with \( s \) odd,  or \( n \) prime to 6.

Suppose \( n = 3s \) with \( s \) odd, so \( f = P_r = \Delta^s \).  Then to attain \( \del(n) \overset{?}{=} 1 \), each of \( \bel(f), \bel(T_3 f), \bel(T_5 f) \) and~\( \bel(T_7 f) \) must attain its maximum value, namely \( \frac{1}{4} \).  By \cref{item:densityleq1/4}, \( \bel(\Delta^s) = \frac{1}{4} \) occurs only for~\( s = 3, 5 \).  For these values \( T_s \Delta^s = \Delta \) and \( \bel(\Delta) = 0 \) forces \( D(n) \leq \frac{3}{4} < 1 \) already.

Likewise, if \( n \) is prime to 6, then \( f = P_n = C^n \).  To attain \( \del(n) \overset{?}{=} 1 \), each of \( \bel(f) \) and \( \bel(T_\ell C^n) \) for~$\ell = 5, 7, 11, 13, 17, 19, 23$ must attain its maximum value, namely \( \frac{1}{8} \).  By \cref{ab9} \( \bel(C^n) = \frac{1}{8} \) occurs only for \( n = 5, 7, 13 \).  For these values \( T_n C^n = C \) and \( \bel(C) = 0 \) forces \( \del(n) \leq \frac{7}{8} < 1 \) already.

\fbox{\( \del(2n) < \frac{1}{2} \)}  Set \( r = 2n \). It's clear from the table that \( \del(r) \leq \frac{1}{2} \), as rows one and five are excluded, since~\( r \) is even.  For~\( \del(r) \mathrel{\smash{\overset{?}{=}}} \frac{1}{2} \) to occur, we must either have: \( r = 6s \) with \( s \) odd, or \( r = 2s \) with \( s \) prime to 6.

Suppose \( r = 6s \) with \( s \) odd, so \( f = P_r = \Delta^s \).  Then to attain \( D(6r) \overset{?}{=} \frac{1}{2} \), both \( \delta(f) \) and \( \bel(T_3 f) \) must attain the value \( \frac{1}{4} \).  Again \cref{item:densityleq1/4} shows \( \bel(\Delta^s) = \frac{1}{4} \) occurs only for \( s = 3, 5 \), then \( T_3 \Delta^5 = T_3 \Delta^3 = \Delta \) and \( \bel(\Delta) = 0 \) immediately gives \( \del(r) = \frac{1}{4} < \frac{1}{2} \).

Likewise, if \( r = 2s \), with \( s \) prime to 6, then \( f = P_r = C^s \).  To attain \( \del(r) \overset{?}{=} \frac{1}{2} \), each of \( \bel(f) \), \( \bel(T_5 f) \), \( \bel(T_7 f) \), \( \bel(T_{11} f) \) must attain the value \( \frac{1}{8} \).  \Cref{ab9} shows \( \bel(C^s) = \frac{1}{8} \) only for \(s = 5, 7, 13\), then \( T_5 C^s = 0\) immediately gives \( \del(r) \leq \frac{3}{8} < \frac{1}{4} \).

\fbox{\( \del(4n) < \frac{1}{4} \)}  Set \( r = 4n \).  It is clear from the table that \( \del(4n) \leq \frac{1}{4} \), as rows one, two, five and six are excluded.  We check each of the remaining possibilities in turn, to see whether and how \( \del(r) \mathrel{\smash{\overset{?}{=}}} \frac{1}{4} \) occurs.

If \( r = 12 s \), with \( s \) odd, then \( f = P_r = \Delta^s \).  We know \( \delta(f) = \frac{1}{4} \) only for \( s = 3, 5 \), which leads to the exceptions \( n = 9, 15 \).

If \( r = 24 s \), then \( f = P_r = \Delta^s \).  We know from \cref{alldeltapowers}, \( \bel(f) = \frac{1}{4} \) only for \( s = 3, 5 \) (also when \( s \) is not odd), which leads to the exceptions \( n = 18, 30 \).

If \( r = 4s \), with \( s \) prime to 6, then \( f = P_r = C^s \).  For \( \del(r) \overset{?}{=} \frac{1}{4} \), both \( \bel(f) \) and \( \bel(T_5 f) \) must attain the value~\( \frac{1}{8} \).  From \cref{ab9}, we know \( \bel(C^s) = \frac{1}{8} \) occurs only for \( s = 5, 7, 13 \), then \( T_5 C^s  = 0 \) immediately gives \( \del(r) = \frac{1}{8} < \frac{1}{4} \).

Finally, if \( r = 8s \) with \( 3 \nmid s \), then \( f = P_r = C^s \).  For \( \del(r) \overset{?}{=} \frac{1}{4} \) both \( \bel(f) \) and \( \bel(U_2 f) \) must attain the value~\( \frac{1}{8} \).  From \cref{allcpowers}, we know \( \bel(C^s) = \frac{1}{8} \) occurs only for \( s = 5, 7, 13 \) (also when $s$ is not coprime to 6), then \( U_2 C^s  = 0 \) immediately gives \( \del(r) = \frac{1}{8} < \frac{1}{4} \).

This completes the proof of \cref{thm:1/2}.  \hfill \qedsymbol{}

\begin{remark}\label{expectations}
Assume Bellaïche's density expectation (see p.~\pageref{18}): namely that for graded nonspecial $f$ we have~$\bel(f) = \frac{1}{8}$. It then follows from \cref{explicit} that $\del(6n) = \frac{1}{4}$, for $n$ odd, whenever both $\Delta^n$ and~$T_3 (\Delta^n)$ are nonspecial.  Likewise $\del(3n) = \frac{1}{2}$, for $n$ odd, whenever  all of  $\Delta^n$, $T_3 (\Delta^n)$, $T_5(\Delta^n)$, $T_7(\Delta^n)$ are nonspecial. Since the space of special forms is ``thin`` in the space of all forms (see the discussion at the end of \cref{tf}), for ``most" odd $n$ we should expect $\del(6n) = \frac{1}{4}$, and $\del(3n) = \frac{1}{2}$. Similarly, one might conjecture that $\bel(f) = \frac{1}{16}$ for nonspecial graded $f \in K(9)$, and that the set of special~$f$ is ``thin". The table in \cref{explicit} then invites guesses about values of $D(r)$ for various $r$.
\end{remark}

\subsection{Proof of \texorpdfstring{\cref{thm:main-thm}}{Theorem \ref{thm:main-thm}}}
Recall the sequences \( z_n = \frac{1}{3}(2 \cdot 4^n + 1) \)
and \( w_n = 4^n + 1\).
We note that infinite families \( z_n, 3z_n \) and \( w_n \)  appearing in \cref{thm:main-thm} correspond to three families of \( m(a,b) \)-basis vectors which are pure powers of \( \Delta \).  More precisely, we have
\begin{align}
    \label{seq:zn} \begin{split}m(2^n-1, 0) &= \Delta^{z_n} \,, \\
    m(2^n, 0) &= \Delta^{3 z_n} \,, \text{ and } \\
    m(0, 2^{n-1}) &= \Delta^{w_n} \,.
    \end{split}
\end{align}
See \cite[\S 5]{NS2} and \cite{Sletter} for the proof of \eqref{seq:zn}.\footnote{
In further unpublished work, Bellaïche has used Derksen's results on zeros of linear recurrences in charactersitic~$p$~\cite{derksen} combined with linear recurrences for the sequences $\{T_\ell(\Delta^n)\}_n$ from \cite[Th\'eor\`eme 3.1]{NS1} to prove that $z_n$, $3z_n$, and $w_n$ are the only powers of~$\Delta$ that are dihedral.}%

    We start by recalling from \cref{explicit}, the following expressions for \( \del \) in terms of \( \bel \), when \( r \) is odd.
        \begin{align*}
            \del(3  r) &= \bel(\Delta^r) + \bel(T_3 \Delta^r) + \bel(T_5 \Delta^r) + \bel(T_7 \Delta^r) \\
            \del(6  r) &= \bel(\Delta^r) + \bel(T_3 \Delta^r) \\
            \del(12  r) &= \bel(\Delta^r)  \\
            \del(24  r) &= \bel(\Delta^r)  
        \end{align*}
        For each case \( r = z_n, 3z_n, w_n \), we will find the relevant images under the Hecke operators, and calculate the resulting densities using \cref{bellaiche-density}.

        \fbox{Case $r = z_n$} From \eqref{seq:zn}, we know \( \Delta^{z_n} = m(2^n-1, 0) \), so that 
        \begin{align*}
            T_3 \Delta^{z_n} &= T_3 \, m(2^n-1, 0) = m(2^n-2, 0) \,, \\
            T_5 \Delta^{z_n} &= T_5 \, m(2^n-1,0) = 0 \\
            T_7 \Delta^{z_n} &= 0 \,.
        \end{align*}
        The first two follow as $m(a,b)$ is a basis adapted to the graded parameters \( (T_3,T_5) \), c.f. \cref{def:adaptedbasis}.  The last follows since \( T_7 \in xy \FF_2\llb x^2,y^2 \rrb \), and \( y = T_5 \) already annihilates \( \Delta^{z_n} \) (see \cite[(7)]{NS2}).  This gives
        \[
            D(3 z_n) = D(6 z_n) \,.
        \]

        We now compute \( \bel(\Delta^{z_n}) = \bel\big(m(2^n-1,0)\big) \) and \( \bel(T_3 \, \Delta^{z_n}) = \bel\big(m(2^n-2,0)\big) \) using \cref{bellaiche-density}.  Recall that if \( u(a) = \text{number of 1s in $a$ base 2} \), and \( v(a) = \text{2-adic valuation of $a$} \), then 
        \[
        \bel\big(m(a,0)\big) = \bel\big(m(0,a)\big) = \frac{1}{2^{u(a) + v(a) + 1}} \,.
        \]
        
        For \( n \geq 1 \), we find
        \[
            u(2^n - 1) = n \,, \quad v(2^n - 1) = 0  \,,
        \]
        hence
        \[
            \bel\big(m(2^n-1, 0)\big) =  \frac{1}{2^{n+1}} \,.
        \] 
        Likewise
        \[
            u(2^n-2) = n-1 \,, \quad v(2^n-2) = \begin{cases}
            n & \text{if $n\geq2$,} \\
            \infty & \text{if $n=1$,}
            \end{cases}
        \]
        hence
        \[
            \bel\big(m(2^n-2,0)\big) = \begin{cases}
                \frac{1}{2^{n+1}} & \text{if $n\geq2$,} \\
                0 & \text{if $n=1$.}
            \end{cases}
        \]
        
        The results for 
        \begin{align*}
        \del(3 z_n) &= \del(6 z_n) = \bel(m(2^n-1,0)) + \bel(m(2^n-2,0)) \text{ and}  \\
          \del(12 z_n) &= \del(24 z_n) = \bel(m(2^n-1,0)) 
          \end{align*} follow directly.

          \fbox{Case \( r = 3z_n \)} From \eqref{seq:zn}, we know \( \Delta^{3z_n} = m(2^n, 0) \).  As before
          \begin{align*}
            T_3 \Delta^{3z_n} = m(2^n-1,0) \,, \\
            T_5 \Delta^{3z_n} = T_7 \Delta^{3z_n} = 0 \,,
          \end{align*}
          so that \( D(3 \cdot 3z_n) = D(6 \cdot 3z_n) \).  We computed \( \bel(T_3\Delta^{3z_n}) = \bel(m(2^n-1,0) \) above, so we only need to compute \( \bel(\Delta^{3z_n}) = \bel(m(2^n, 0)) \).  

          We find
          \[
                u(2^n) = 1 \,, v(2^n) = n \,,
          \]
          so that by \cref{bellaiche-density}
          \[
                \bel(m(2^n,0)) = \frac{1}{2^{n+2}} \,.
          \]
          
        The results for 
        \begin{align*}
        \del(3 \cdot 3z_n) &= \del(6 \cdot 3z_n) = \bel(m(2^n,0)) + \bel(m(2^n-1,0)) \text{ and}  \\
          \del(12 \cdot 3z_n) &= \del(24 \cdot 3z_n) = \del(m(2^n,0)) 
          \end{align*} follow directly.

          \fbox{Case \( r = w_n \)} From \eqref{seq:zn}, we know \( \Delta^{w_n} = m(0, 2^{n-1}) \).  This time, we compute
          \begin{align*}
            T_3 \Delta^{w_n} &= T_3 m(0, 2^{n-1}) = 0 \,, \\
            T_5 \Delta^{w_n} &= T_5 m(0, 2^{n-1}) = m(0, 2^{n-1}-1) \,, \\
            T_7 \Delta^{w_n} &= 0 \,,
          \end{align*}
          as \( T_7 \in x y \FF_2 \llb x^2,y^2 \rrb \) (see \cite[(7)]{NS2}), and already \( x = T_3 \) annihilates \( \Delta^{w_n} \).  This means \[
           \del(6 w_n) = \del(12 w_n) = \del(24 w_n) \,.
           \]
           Again using \cref{bellaiche-density}, we can compute \( \bel(\Delta^{w_n}) = \bel\big(m(0, 2^{n-1})\big) \) and \( \bel(T_5 \Delta^{w_n}) = \bel\big(m(0, 2^{n-1}-1)\big) \).

           We have 
           \[
                u(2^{n-1}) = 1 \,, \quad v(2^{n-1}) = n-1 \,,
           \]
           so that
           \[
                \bel\big(m(0, 2^{n-1})\big) = \frac{1}{2^{n+1}}.
           \]
           Likewise
           \[
                u(2^{n-1}-1) = n-1 \,, \quad v(2^{n-1}-1) = \begin{cases}
                    0 & \text{if $n\geq2$,} \\
                    \infty & \text{if $n=1$,}
                \end{cases}
           \]
           and hence
           \[
            \bel(m(0,2^{n-1}-1)) = \begin{cases}
                    \frac{1}{2^{n}} & \text{if $n\geq2$,} \\
                    0 & \text{if $n=1$.}
                \end{cases}
           \]

            The results for 
        \begin{align*}
        \del(3 w_n) &= \bel\big(m(0,2^{n-1})\big) + \bel\big(m(0,2^{n-1}-1)\big) \text{ and}  \\
           \del(6 w_n)  &= \del(12 w_n) = \del(24 w_n) = \bel\big(m(0,2^{n-1})\big)
          \end{align*} follow directly.
          \hfill \qedsymbol{}

\subsection{Computations for \texorpdfstring{\cref{thm:abelian-thm}}{Remark \ref{thm:abelian-thm}}}
\label{proof:rk:abelian}

As already noted in \cref{thm:abelian-thm}, one can carry out the calculations of \( D(3r), D(6r), D(12r) \) and \( D(24r) \) in a few further cases, namely for abelian $\Delta$ powers.  Referring to the table and remark thereafter in \cite[end of \S3.2]{Bmem}, the forms \(\Delta^7\),~\(\Delta^{19} \)~and~\(\Delta^{21}\) are abelian, with no higher pure $\Delta$ powers expected to be abelian.  (The case \( r = 1 \) giving \( D(3) = D(6) = D(12) = D(24) = 0 \) is handled as part of \cref{thm:zero-density}, while \( r = 3 \) and \( r = 5 \) covered as the \( n=1 \) cases of \cref{thm:main-thm}.)

We treat the case \( r = 7 \) extracting the relevant information from the table \cite[end of \S3.2]{Bmem}.
The cases~\(r = 19, 21\) are analogous.  We again recall from \cref{explicit} that
\begin{align*}
    \del(3 r) &= \bel(\Delta^r) + \bel(T_3 \Delta^r) + \bel(T_5 \Delta^r) + \bel(T_7 \Delta^r) \\
    \del(6 r) &= \bel(\Delta^r) + \bel(T_3 \Delta^r) \\
    \del(12 r) &= \bel(\Delta^r)  \\
    \del(24 r) &= \bel(\Delta^r)  
\end{align*}

\fbox{Case \( r = 7 \)}  From the table \cite[end of \S3.2]{Bmem}, we directly have
\[
    T_p \Delta^7 = \begin{cases}
        0 & \text{if $p \equiv 1 \pmod{8}$,} \\
        \Delta^5 & \text{if $p \equiv 3 \pmod{8}$,} \\
        \Delta^3 & \text{if $p \equiv 5 \pmod{8}$,} \\
        \Delta & \text{if $p \equiv 7 \pmod{16}$,} \\
        0 & \text{if $p \equiv 15 \pmod{16}$.}
    \end{cases}
\]
We have that \( a_p(\Delta^7) = a_1(T_p \Delta^7) = 1 \) if and only if \( p \equiv 7 \pmod{8} \).  Since primes \( p \) equidistribute in the~\( \varphi(16) = 8 \) residue classes modulo 16 (here $\varphi$ is the Euler totient function), we deduce directly that
\[
    \bel(\Delta^7) = \frac{1}{\varphi(16)} = \frac{1}{8} \,.
\]
Likewise \( T_3 \Delta^7 = \Delta^5 \), with~\( \bel(\Delta^5)  = \frac{1}{4} \).  We also have \( T_5 \Delta^7 = \Delta^3 \), with \( \bel(\Delta^3)  = \frac{1}{4} \), and \( T_7 \Delta^7 = \Delta \), with \( \bel(\Delta) = 0 \).  (Either use \cref{bellaiche-density} and that \( m(0,0) = \Delta \), \( m(1,0) = \Delta^3\), and \( m(0,1) = \Delta^5 \), or read these results off from the table in \cite[end of \S3.2]{Bmem}.)

Overall
\begin{align*}
    \del(3\cdot7) &= \bel(\Delta^7) + \bel(T_3 \Delta^7) + \bel(T_5 \Delta^7) + \bel(T_7 \Delta^7) = \frac{1}{8} + \frac{1}{4} + \frac{1}{4} + 0 = \frac{5}{8} \\
    \del(6\cdot7) &= \bel(\Delta^7) + \bel(T_3 \Delta^7)  = \frac{1}{8} + \frac{1}{4} = \frac{3}{8} \\
    \del(12\cdot7) &= \del(24\cdot7) = \bel(\Delta^7) = \frac{1}{8} \,. 
\end{align*}

Likewise, in the cases \( D(s), D(2s), D(4s) \), with \( s \) prime to 6, and \( D(8s) \), with \( 3 \nmid s \), we can obtain results from~\( C^5\),~\(C^7\),~\(C^{13}\) the three abelian powers of \( C \) introduced in \cref{ab9}.

\fbox{Level 9} In the level 9 case, we note that for \( \ell \in  (\Z/24\Z)^\times = \{ 1, 3, 7, 11, 13, 17, 19, 23 \} \), and \( s \in \{ 5, 7, 13 \} \)
\[
    T_\ell C^s = \begin{cases}
        C & \text{if \( \ell = s \)} \\
        0 & \text{otherwise} \,.
    \end{cases}
\]
Moreover \( U_2\, C^\text{odd} = 0 \).  So from \cref{explicit}, every term in \( D(r) \) except \( \delta(f) \) vanishes.  Hence for \( s = 5, 7, 13 \), we find \[
    D(s) = D(2s) = D(4s) = D(8s) = \delta(C^s) = \frac{1}{8} \,,
\]
where the explicit description in \cref{ab9} shows that \( \delta(C^s) = \frac{1}{8} \), for these \( s \).

\appendix

\section*{Appendix: Complements}\label[app]{complements}

In this section, we give some additional background, private motivation, and experimental observations in the direction of the parity of coefficients of modular forms, and the parity of the partition function in particular.\medskip

The famous aforementioned problem about the parity of the partition function was emphasized again in a talk of Ken Ono at the \href{http://www.mi.uni-koeln.de/Bringmann/Konferenz2024.html}{\emph{International Conference on Forms and $q$-Series}} (11--15 March 2024, Cologne).  During the extra-long commute back to Bonn\footnote{via tram, due to inopportunely scheduled railway construction work} the last two authors, together with Pieter Moree, conducted some numerical experiments and plotted the parity of the partition function as a random walk.  These computations were then extended by the first two authors to \( n \leq 2 \cdot 10^8 \) with some mildly optimized implementation of the pentagonal number theorem in \texttt{C}.

Under the hypothesis that the \( i^{\rm th} \) partition value \( p(i) \) is randomly distributed even or odd with equal probability independent from all other values, one can model the random walk as a family of independent identically distributed random variables \( P_i \), where \( P_i \) has distribution: \( -1 \) with probability \( \frac{1}{2} \) and \( 1 \) with probability \( \frac{1}{2} \).  One computes the variance and expectation (mean) to be \( \Var(P_i) = 1 \) and \( E(P_i) = 0 \), so that (by their assum\`ed independence), the Central Limit Theorem implies \( W \coloneqq \frac{1}{n} \sum_{i=1}^n P_i \) is distributed normally as \( \mathcal{N}(0, \frac{1}{n}) \) with mean 0 and variance \( \frac{1}{n} \).  With ${\sim}66\%$ probability, \( W \) is then within 1 standard deviation of the mean i.e. \( 0 \pm \frac{1}{\sqrt{n}} \) and with ${\sim}95\%$ probability is within 2 standard deviations i.e. \( 0 \pm \frac{2}{\sqrt{n}} \) of the mean.  After rescaling \( n W = \sum_{i=1}^n P_i \) should be within \( \pm \sqrt{n} \) of 0 with ${\sim}66\%$ probability, and within~\( \pm 2 \sqrt{n} \) of 0 with ${\sim}95\%$ probability.

The random walk of \( p(i) \) and its position relative to the curves \( 0\pm \sqrt{n} \) and \( 0 \pm 2\sqrt{n} \) (indicating the $\sim66\%$ and $\sim95\%$ probability regions) is shown in \cref{fig:pwalk}.

\begin{figure}[p]
     \includegraphics[width=0.7\textwidth]{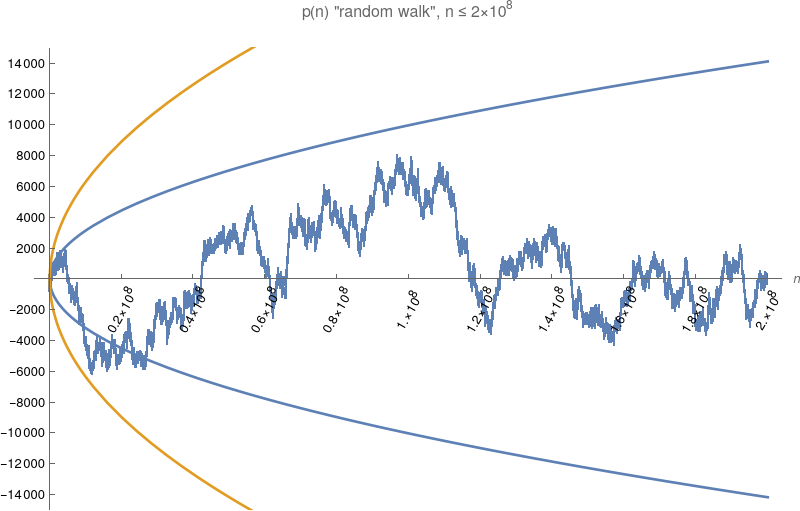}     
     \caption{Random walk obtained from parity of the partition function, moving up for~\( p(i) \) even, and down for \( p(i) \) odd.  More precisely, the plot shows \( (n, \sum_{i\leq n} (-1)^{p(i)}) \), for \( 1 \leq n \leq 2\cdot10^8 \).  The blue curve is \( 0 \pm \sqrt{n} \), and the orange curve is \( 0 \pm 2 \sqrt{n} \), indicating the regions 1 and 2 standard deviations away from the mean, with probabilities $\sim66\%$ and $\sim95\%$ respectively.}
     \label{fig:pwalk}
\end{figure}

\begin{figure}[p]
     \includegraphics[width=0.7\textwidth]{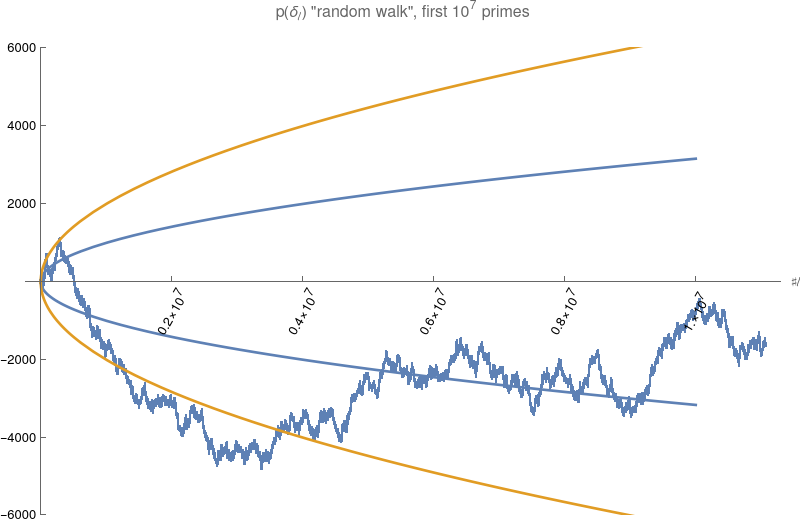}     
     \caption{Random walk obtained from parity of the partition function, moving up for~\( p(\delta_{\pi_i}) \) even, and down for \( p(\delta_{\pi_i}) \) odd.  More precisely, the plot shows \( (n, \sum_{i\leq n} (-1)^{p(\delta_{\pi_i})}) \), for \( 1 \leq n \leq 2\cdot10^8 \).  The blue curve is \( 0 \pm \sqrt{n} \), and the orange curve is \( 0 \pm 2 \sqrt{n} \), indicating the regions 1 and 2 standard deviations away from the mean, with probabilities $\sim66\%$ and $\sim95\%$ respectively.}
     \label{fig:deltawalk}
\end{figure}

Ono suggested looking for a parity bias in special subsequences of the partition function values, as a way to investigate the Partition Parity Conjecture and to potentially construct sequences with known parity.  Because of its importance in the theory of partition congruences, Ono specifically suggested looking for a bias in the random walk of the special subsequence  \( \delta_\ell \), for \( \ell \geq 5 \) prime, defined by \( 0 < \delta_\ell < \ell,\)  \(\delta_\ell \equiv 24^{-1} \pmod{\ell} \).  The first few terms of this sequence are enumerated below.
\begin{table}[H]
    \begin{tabular}{c|cccccccccccc}
$\ell$ & 5 & 7 & 11 & 13 & 17 & 19 & 23 & 29 & 31 & 37 & 41 & 43 \\ \hline
$\delta_\ell \equiv 24^{-1} \pmod{\ell} $ & 4 & 5 & 6 & 6 & 5 & 4 & 1 & 23 & 22 & 17 & 12 & 9
\end{tabular}
\label{tbl:deltaell}
\end{table}

Formally, let \( \pi_i \) be the \( i \)th prime at least 5, so \( \pi_1 = 5 \), \( \pi_2 = 7 \), \( \pi_3 = 11 \), and plot the random walk obtained from \( p(\delta_{\pi_i}) \).  The result is given in \cref{fig:deltawalk}.

Up to these points, both graphs are consistent with the hypothesis that the values \( p(i) \), respectively \( p(\delta_{\pi_i}) \), are randomly and independently distributed odd and even, with equal probability.  Although the random walk of \( p(\delta_{\pi_i}) \) does remain negative for a rather long stretch (between \( i \approx 0.1 \times 10^7 \) and \( 1.1 \times 10^7 \) in the graph in \cref{fig:deltawalk}), there is no \emph{obvious} bias, and it is entirely plausible that the walk will go positive if the graph is further extended. (It already gets quite close to crossing the horizontal axis around \( i \approx 1.0 \times 10^7 \)).  One can also observe that the graph in \cref{fig:deltawalk} strays outside the the one standard-deviation region until quite late.  Note however that there is an order of magnitude difference between the scales in  \cref{fig:pwalk} and  \cref{fig:deltawalk}, and this behaviour can be seen in the first graph until the first axis tick at around \( i \approx 0.2 \times 10^8 \).

% to force bibliography after figures
\clearpage

\bibliographystyle{alpha}

\parskip=0pt           % 15 pt spacing between paragraphs

\end{document}